\documentclass[12pt]{amsart}
\usepackage{amscd,setspace,amssymb,amsopn,amsmath,amsthm,mathrsfs,lmodern,graphics,amsfonts,enumerate,verbatim,calc
} 
\usepackage[all]{xy}
\usepackage[colorlinks=true, citecolor=PineGreen, filecolor=PineGreen, linkcolor=Mahogany,pagebackref, hyperindex]{hyperref}
\usepackage[dvipsnames]{xcolor}
\usepackage{todonotes}
\usepackage{color, hyperref}
\usepackage[OT2,OT1]{fontenc}

\newcommand\cyr{%
\renewcommand\rmdefault{wncyr}%
\renewcommand\sfdefault{wncyss}%
\renewcommand\encodingdefault{OT2}%
\normalfont
\selectfont}
\DeclareTextFontCommand{\textcyr}{\cyr}
\usepackage{amssymb,amsmath}
\DeclareFontFamily{OT1}{rsfs}{}
\DeclareFontShape{OT1}{rsfs}{n}{it}{<-> rsfs10}{}
\DeclareMathAlphabet{\fmathscr}{OT1}{rsfs}{n}{it}
\numberwithin{equation}{section}
\newtheorem{Theoremx}{Theorem}
 
\newtheorem{theorem}{Theorem}[section]
\newtheorem{lemma}[theorem]{Lemma}
\newtheorem{proposition}[theorem]{Proposition}
\newtheorem{corollary}[theorem]{Corollary}

\newtheorem{conjecture}[theorem]{Conjecture}
\theoremstyle{definition}
\newtheorem{definition}[theorem]{Definition}
\newtheorem{remark}[theorem]{Remark}
\theoremstyle{remark}

\newtheorem{example}[theorem]{Example}

\newcommand{\Ass}{\operatorname{Ass}}
\newcommand{\im}{\operatorname{Im}}

\newcommand{\Spec}{\operatorname{Spec}}

\newcommand{\Ht}{\operatorname{ht}}

\newcommand{\id}{\operatorname{id}}

\newcommand{\Supp}{\operatorname{Supp}}

\newcommand{\Att}{\operatorname{Att}}
\newcommand{\Ann}{\operatorname{Ann}}

\newcommand{\fm}{\mathfrak{m}}
\newcommand{\fp}{\mathfrak{p}}
\newcommand{\fq}{\mathfrak{q}}
\newcommand{\fa}{\mathfrak{a}}

\newcommand{\FF}{\mathbb{F}}

\newcommand{\NN}{\mathbb{N}}

\begin{document}
\title[Nilpotence of Frobenius action on local cohomology]{Nilpotence of Frobenius actions on local cohomology and Frobenius closure of ideals}

\author[Thomas Polstra]{Thomas Polstra}
\thanks{Polstra was supported in part by NSF Postdoctoral Research Fellowship DMS $\#1703856$.}
\address{Department of Mathematics, University of Utah, Salt Lake City, UT 84102 USA}
\email{polstra@math.utah.edu}

\author[Pham Hung Quy]{Pham Hung Quy}
\thanks{ Quy is partially supported by the Vietnam National Foundation for Science and Technology Development (NAFOSTED) under grant number 101.04-2017.10.}
\address{Department of Mathematics, FPT University, and Thang Long Institute of Mathematics and Applied Sciences, Ha Noi, Viet Nam}
\email{quyph@fe.edu.vn}


\begin{abstract} The study of Frobenius actions on local cohomology modules over a local ring of prime characteristic has interesting connections with the theory of tight closure. This paper establishes new connections by developing the notion of relative Frobenius actions on local cohomology. As an application, we show that a ring has $F$-nilpotent singularities if and only if the tight closure of every parameter ideal is equal to its Frobenius closure.
\end{abstract}

\dedicatory{Dedicated to Professor Ngo Viet Trung on the occasion of his sixty-fifth birthday.}

 \maketitle

\section{Introduction}

Denote by $(R,\fm,k)$ a commutative Noetherian local ring with unique maximal ideal $\fm$ and residue field $k$. Unless otherwise stated, we assume $R$ is of prime characteristic $p>0$. We let $F^e:R\to R$ denote the $e$th-iterate of the Frobenius endomorphism which maps an element $r\mapsto r^{p^e}$. The Frobenius endomorphisms induce natural Frobenius actions of local cohomology modules
\[
F^e:H^i_\fm(R)\to H^i_\fm(R)
\]
for each $i\in \NN$. There are several interesting classes of singularities which can be defined in terms of the behavior of these Frobenius actions. These include $F$-rational, $F$-injective, and $F$-nilpotent singularities. The ring $R$ is $F$-injective if the Frobenius actions above are all injective. We say $R$ is an $F$-nilpotent ring if the above Frobenius actions are nilpotent on the local cohomology modules $H^i_\fm(R)$ when $i<\dim(R)$ and the nilpotent submodule of $H^d_\fm(R)$ is ``as large as possible.'' We refer the reader to Section~\ref{Preliminary Section} for a precise definition. An excellent equidimensional ring $R$ is $F$-rational if and only if it is both $F$-injective and $F$-nilpotent. All three notions have interesting and important connections with the theory of tight closure.

Let $R^{\circ}$ be the multiplicative set of elements of $R$ which are not contained in a minimal prime ideal. Given an ideal $I\subseteq R$ we let $I^{[p^e]}$ denote the expansion of $I$ along $F^e$. If $x\in R$ then we say that $x$ is in the tight closure of $I$ if there exists $c\in R^\circ$ such that $cx^{p^e}\in I^{[p^e]}$ for all $e\gg 0$. The collection of all such elements is denoted by $I^*$. If the above element $c\in R^\circ$ can be taken to be the element $1$, i.e., if $x^{p^e}\in I^{[p^e]}$ for all, equivalently for some, $e\gg 0$, then it is said that $x$ is in the Frobenius closure of $I$. The Frobenius closure of an ideal is denoted by $I^F$. The sets $I^F$ and $I^*$ are ideals of $R$ and there are inclusions $I\subseteq I^F\subseteq I^*$. If $I=I^F$ then $I$ is said to be Frobenius closed and if $I=I^*$ then $I$ is called tightly closed. We refer the reader to \cite{HH90, HH94, H96} for the basics of tight closure.

A parameter ideal of $R$ is an ideal $\fq$ which is generated by a full system of parameters for $R$. The ring $R$ is $F$-rational if and only if $\fq=\fq^*$ for every parameter ideal of $R$. The second author and Shimomoto show in \cite{QS17} that if $\fq=\fq^F$ for each parameter ideal of $R$ then $R$ is $F$-injective, but they also show that not every parameter ideal need be Frobenius closed in an $F$-injective ring. In this article, the following classification of $F$-nilpotent rings is given:

\begin{Theoremx}\label{Main theorem frobenius closure} Let $(R,\fm,k)$ be an excellent equidimensional local ring of prime characteristic $p>0$. Then the following are equivalent:
\begin{enumerate}
\item The ring $R$ is $F$-nilpotent.
\item $\fq^F=\fq^*$ for every parameter ideal $\fq$ of $R$.
\end{enumerate}
\end{Theoremx}

The example of an $F$-injective ring with a parameter ideal not being Frobenius closed given in \cite{QS17} is not equidimensional. However, we shall see that all $F$-nilpotent rings are equidimensional, and so Theorem~\ref{Main theorem frobenius closure} might be an indication that all parameter ideals in an equidimensional $F$-injective ring are Frobenius closed.

If $x_1,\ldots, x_d$ is a system of parameters for $R$ then the top local cohomology module with support in the maximal ideal has a very explicit description:
\[
H^d_\fm(R)\cong \varinjlim_N\left(\cdots \to \frac{R}{(x_1^N,\ldots, x_d^N)}\xrightarrow{ x_1\cdots x_d}\frac{R}{(x_1^{N+1},\ldots, x_d^{N+1})}\to \ldots\right).
\]
Furthermore, if $R$ is Cohen-Macaulay, i.e., $H^i_\fm(R)=0$ for all $i<\dim (R)$, then the maps in direct limit system are injective. Having an explicit discription of $H^d_\fm(R)$ provides a great advantage in making connections between the behavior of Frobenius actions on local cohomology modules and prime characteristic properties of parameter ideals in rings which are assumed to be Cohen-Macaulay. For example, an important open problem, which is solved in the Cohen-Macaulay case, is whether $F$-injectivity deforms. That is if $x\in R$ is a regular element such that $R/(x)$ is an $F$-injective ring then is it necessarily the case that $R$ is an $F$-injective as well? Fedder proved that $R$ is indeed $F$-injective under the assumptions $R$ is Cohen-Macaulay and $R/(x)$ is $F$-injective for some regular element $x$ in \cite{F83}. We refer the reader to \cite{HMS14} and \cite{MQ16} for more recent developments on the deformation of $F$-injectivity problem.

Similar to the study of $F$-injective rings, the difficulties of understanding $F$-nilpotent rings comes from the study of non-Cohen-Macaulay rings. The notion of a filter regular element allows some insight to the behavior of nonzero lower local cohomology modules. An element $x\in \fm$ is called a filter regular element if $x$ avoids all non-maximal associated primes of $R$. The $e$th-iterate of the Frobenius endomorphism $F^e:R/(x)\to R/(x)$ can be factored as $R/(x)\to R/(x^{p^e})\to R/(x)$ where the second map is the natural projection. In particular, there are induced maps of local cohomology modules $F^e_R: H^i_\fm(R/(x))\to H^i_\fm(R/(x^{p^e}))$. We say that $R/(x)$ is $F$-nilpotent relative to $R$ if for each $i<\dim(R/(x))$ and $\eta\in H^i_\fm(R/(x))$ there exists $e\in \NN$ such that $F^e_R(\eta)=0$ as an element of $H^i_\fm(R/(x^{p^e}))$. We will also require that the relative nilpotent submodule of $H^{d-1}_\fm(R/(x))$ be as ``large as possible.'' We refer the reader to Section~\ref{Section F-nilpotent}  for a precise definition. Our proof of Theorem~\ref{Main theorem frobenius closure} will depend on the following characterization of $F$-nilpotent rings.

\begin{Theoremx}\label{Main theorem deformation} Let $(R,\fm,k)$ be a excellent and equidimensional local ring  of prime characteristic $p>0$. Then the following are equivalent:
\begin{enumerate}
\item The ring $R$ is $F$-nilpotent.
\item For each filter regular element $x$ on $R$ the cyclic module $R/(x)$ is $F$-nilpotent with respect to $R$.
\end{enumerate}
\end{Theoremx}

Theorem~\ref{Main theorem deformation} is crucial to proof of Theorem~\ref{Main theorem frobenius closure}, but more is needed. It is not enough to consider the algebraic properties of a filter regular element, we will need to understand algebraic properties of a filter regular sequence. A sequence of elements $x_1,\ldots, x_\ell$ is a called a filter regular sequence if for each $1\leq j\leq \ell$ the class of $x_j$ is a filter regular element of $R/(x_1,\ldots,x_{j-1})$.  We show that ideals generated by a filter regular sequence enjoy the following desirable property:

\begin{Theoremx}\label{Main theorem LC} Let $(R,\fm,k)$ be a local ring of prime characteristic $p>0$. Let $x_1,\ldots,x_t$ be a filter regular sequence of $R$ and $I=(x_1,\ldots,x_t)$. Then there exists a positive integer $C$ such that for all $e\in \NN$ 
\[
\fm^{Cp^e}\cdot H^0_\fm(R/I^{[p^e]})=0.
\]
\end{Theoremx}

Understanding algebraic and cohomological properties of filter regular sequences yields the following characterization of lower local cohomology modules being $F$-nilpotent:

\begin{Theoremx}\label{Main theorem F-nilpotent filter regular sequence} Let $(R,\fm,k)$ be a local ring of dimension $d$ and of prime characteristic $p>0$ and let $t<d$. Then the following are equivalent:
\begin{enumerate}
\item $H^i_\fm(R)$ is $F$-nilpotent for all $i\leq t$.
\item For every filter regular sequence $x_1,\ldots,x_t$ we have 
\[
\bigcup_{n\in \NN} (x_1,\ldots, x_t):\fm^n\subseteq (x_1,\ldots,x_t)^F.
\]
\end{enumerate}
\end{Theoremx}

The paper is organized as follows: Section~\ref{Preliminary Section} covers the basic notions and background material relevant to the results of later sections and develops a few new results concerning $F$-nilpotent rings. Section~\ref{LC section} is where we prove Theorem~\ref{Main theorem LC}. The proof of Theorem~\ref{Main theorem F-nilpotent filter regular sequence} and other characterizations of lower local cohomology modules to be $F$-nilpotent can be found in Section~\ref{Deformation Section}. The final section, Section~\ref{Section F-nilpotent}, is where we piece together the results of the previous sections and prove Theorem~\ref{Main theorem frobenius closure} and Theorem~\ref{Main theorem deformation}.

\section{Preliminaries}\label{Preliminary Section}

\subsection{Local cohomology}
Let $R$ be a Noetherian ring, not necessarily of prime characteristic, $M$ an $R$-module, and $I$ an ideal of $R$. Then we denote by $H^i_I(M)$ the i-th local cohomology module with support at $I$ (\cite{BS13} and \cite{ILL07}). Recall that $H^i_I(M)$ may be computed as the $i$th cohomology of the \v{C}ech complex
\[
\check{C}^\bullet(x_1,\ldots,x_t;M): 0 \to M \to \bigoplus_{i=1}^t M_{x_i} \to \cdots \to M_{x_1 \ldots x_t} \to 0
\]
where $x_1,\ldots,x_t$ are any choice of generators of $I$ up to radical. If $R\to S$ is a homomorphism of rings then for each $R$-module $M$ there is functorial map of local cohomology modules $H^i_I(M)\to H^i_{IS}(M\otimes_R S)$. In particular, if $R$ is of prime characteristic $p>0$, then the $e$th-iterate of the Frobenius endomorphism $F^e:R \to R$ induces the $e$th-Frobenius action $F^e:H^i_I(R) \to H^i_{I^{[p^e]}}(R) \cong H^i_{I}(R)$ with the isomorphism of local cohomology modules coming from the observation that the ideals $I$ and $I^{[p^e]}$ are the same up to radical.

If $c\in \NN$ then $(x_1,\ldots, x_t)$ is cofinal with $(x_1^c,\ldots,x_t^c)$ and $H^t_I(R)$ is isomorphic to the direct limit system
\[
\varinjlim_n \left( \cdots \to\frac{R}{(x^{cn}_1,\ldots, x^{cn}_t)}\xrightarrow{(x_1\cdots x_t)^c} \frac{R}{(x^{c(n+1)}_1,\ldots, x^{c(n+1)}_t)}\to \cdots\right).
\]
In particular, if $R$ is of prime characteristic then the Frobenius action on the top local cohomology module $H^t_I(R)$ has a the following explicit description: If $\eta+(x_1^m,\ldots, x_t^m)$ is a representative of an element in 
\[
H^t_I(R)\cong \varinjlim_n \frac{R}{(x_1^n,\ldots, x_t^n)}
\]
then the $e$th-Frobenius action on $H^t_I(R)$ sends $\eta+(x_1^m,\ldots, x_t^m)\mapsto \eta^{p^e}+(x_1^{p^em},\ldots, x_t^{p^em})$ in $H^t_I(R)$ as it is realized as the direct limit
\[
H^t_I(R)\cong \varinjlim_n \frac{R}{(x_1^{p^en},\ldots, x_t^{p^en})}.
\]

\subsection{Filter regular sequences}
\begin{definition}
Let $M$ be a finitely generated module over a local ring $(R,\fm,k)$ and let $x_1,\ldots,x_t$ be a set of elements of $R$. Then we say that $x_1,\ldots,x_t$ is a \textit{filter regular sequence} on $M$ if the following conditions hold:
\begin{enumerate}
\item $(x_1,\ldots,x_t) \subseteq \fm$.

\item $x_i \notin \mathfrak{p}$ for all $\fp \in \Ass_R\Big(\dfrac{M}{(x_1,\ldots,x_{i-1})M}\Big) \setminus \{\fm\},~i=1,\ldots,t$.
\end{enumerate}
\end{definition}

The notion of a filter regular sequence was introduced by Cuong, Schenzel, and Trung in \cite{STC78}. If $M$ is a finitely generated module over a local ring $(R,\fm,k)$ then a simple prime avoidance argument shows that there always exists a system of parameters for $M$ which is also a filter regular sequence on $M$. Filter regular sequences are also called \textit{$\fm$-filter regular sequence} in other sources.

\begin{lemma}
\label{filter}
Let $(R,\fm,k)$ be a local ring and $M$ a finitely generated $R$-module. Then $x_1,\ldots, x_t \in \fm$ is a filter regular sequence on $M$ if and only if one of the following four equivalent conditions holds:
\begin{enumerate}
\item
For each $1\leq i \leq t$ the quotient
\[
\dfrac{\big((x_1,\ldots,x_{i-1})M:_M x_i\big)}{(x_1,\ldots,x_{i-1})M}
\]
is an $R$-module of finite length.

\item
For each $1\leq i \leq t$ the sequence
\[
\frac{x_1}{1},\frac{x_2}{1},\ldots,\frac{x_i}{1}
\]
forms an $R_{\fp}$-regular sequence in $M_{\fp}$ for every $\fp \in \big(\Spec (R/(x_1,\ldots,x_i)) \cap \Supp_R M\big) \setminus \{\fm\}$.

\item
The sequence $x_1^{n_1},\ldots,x_t^{n_t}$ is a filter regular sequence for all $n_1, \ldots, n_t \ge 1$.

\item The sequence $x_1\ldots, x_t\in \widehat{R}$ is a filter regular sequence of $\widehat{M}$.

If $x_1,\ldots, x_t$ satisfies any of the above equivalent conditions then it also satisfies the following fifth condition:

\item For all  $n_1, \ldots, n_t \ge 1$ we have
\[
\mathrm{Ass}_R(M/(x_1, \ldots, x_t)M) \cup \{\fm\} = \mathrm{Ass}_R(M/(x_1^{n_1}, \ldots, x_t^{n_t})M) \cup \{\fm\}.
\]
\end{enumerate}
\end{lemma}

\begin{proof}
Equivalence of the first three properties can be found in \cite[Proposition~2.2]{NS94}. For the equivalence of (4) with the first three properties we refer to \cite[Remark 2.4]{Q13}. Lastly, observe that the quotient module in condition (1) having finite length is unaffected by completion. Hence condition (5) is equivalent to the other four conditions.
\end{proof}

The following result will be useful to this paper (cf. \cite[Proposition~3.4]{NS94}). 


\begin{lemma}[Nagel-Schenzel isomorphism]
\label{nagel-shenzel}
Let $(R,\fm,k)$ be a local ring and let $M$ be a finitely generated $R$-module. If $x_1,\ldots,x_t$ is a filter regular sequence on $M$ then
\[
H^i_{\fm}(M) \cong
\begin{cases}
H^i_{(x_1,\ldots, x_t)}(M) &\text{ if } i<t\\
H^{i-t}_{\fm}(H^t_{(x_1, \ldots, x_t)}(M)) &\text{ if } i \ge t.
\end{cases}
\]
\end{lemma}

The Nagel-Schenzel's isomorphism allows us to identify $H^t_{\fm}(R)$ as the collection of elements of $H^t_{(x_1, \ldots, x_t)}(R)$ annihilated by a sufficiently large power of the maximal ideal. This is often useful since the local cohomology module $H^t_{(x_1, \ldots, x_t)}(R)$ has an explicit description via \v{C}ech cohomology. Moreover, if $R$ is of prime characteristic then the Frobenius action on $H^t_\fm(R)$ is the restriction of the Frobenius action on $H^t_{(x_1,\ldots, x_t)}(R)$ to $H^0_\fm(H^t_{(x_1,\ldots, x_t)}(R))$.

If $(R,\fm,k)$ is a local ring and $M$ a module with support only at the maximal ideal then $H^i_\fm(M)=0$ for all $i\geq 1$. In particular, if $x$ is a filter regular element of $R$ then $(0:x)$ is supported only at the maximal ideal and examination of the long exact sequence of local cohomology modules induced from the short exact sequence
\[
0\to (0:x) \to R\to R/(0:x)\to 0
\]
shows $H^i_\fm(R/(0:x))\cong H^i_\fm(R)$ for each $i\geq 1$. Furthermore, for each $i\geq 1$ the induced map of local cohomology modules $H^i_\fm(R)\cong H^i_\fm(R/(0:x))\to H^i_\fm(R)$ derived from the short exact sequence
\[
0\to R/(0:x)\xrightarrow{x} R\to R/(x)\to 0
\]
is multiplication by the element $x$.

Another useful property of filter regular elements is the following:

\begin{lemma}\label{0th local cohomology lemma} Let $(R,\fm,k)$ be a local ring and fix $N\in \NN$ large enough so that $H^0_\fm(R)= (0: \fm^N)$. If $x\in \fm^N$ is a filter regular element then $H^0_\fm(R/(0:x))=0$, i.e., $R/(0:x)$ has positive depth. 
\end{lemma}

\begin{proof} Consider the induced long exact sequence of local cohomology modules induced from the short exact sequence
\[
0\to (0:x)\to R\to R/(0:x)\to 0.
\]
Observe that $H^0_\fm((0:x))\to H^0_\fm(R)$ is an isomorphism since every element of $H^0_\fm(R)=(0:\fm^C)$ is an element of $(0:x)$. Furthermore, $H^1_\fm((0:x))=0$ since $(0:x)$ has finite length and therefore $H^0_\fm(R/(0:x))=0$.
\end{proof}

\subsection{Tight closure and Frobenius closure}
 Let $R$ be a Noetherian ring of prime characteristic $p>0$. We let $F^e_*(R)$ denote the $R$-bimodule which as an Abelian group and as a right $R$-module is $R$, but as a left module $F^e_*(R)$ is the module obtained by restricting scalars under $F^e:R\to R$. Given $x\in R$ we let $F^e(x)$ denote the corresponding element in $F^e_*(R)$. Thus given $c\in R$ we have $F^e_*(x)c=F^e_*(xc)$ and $cF^e_*(x)=F^e_*(c^{p^e}x)$. The ring $R$ is $F$-finite if $F^e_*(R)$ is a finite left $R$-module for some, equivalently for all, $e\in \NN$.

\begin{definition}[\cite{HH90,HH94,H96}] Let $\displaystyle R^{\circ} = R \setminus \bigcup_{\fp \in \mathrm{Min}R} \fp$ and $I\subseteq R$ an ideal of $R$.
\begin{enumerate}
  \item The Frobenius closure of $I$ is the ideal 
  \[
  I^F = \{x \mid  x^{p^e} \in I^{[p^e]} \text{ for some } e\in \NN\}.
  \]
  \item The tight closure of $I$ is the ideal 
  \[
  I^* = \{x \mid cx^{p^e} \in I^{[p^e]} \text{ for some } c \in R^{\circ} \text{ and for all } e \gg 0\}.
  \]
\end{enumerate}
\end{definition}
An element $x \in I^F$ if and only if $x$ is in the kernel of the composition of maps
 $$R \to R/I  \xrightarrow{\mathrm{id} \otimes F^e} R/I \otimes F_*^e(R)$$
 for some $e \ge 0$. Similarly, an element $x \in I^*$ if and only if it contained in the kernel of the composition
\[
R \to R/I \to R/I \otimes F_*^e(R)\xrightarrow{\mathrm{id} \otimes F_*^e(c)} R/I \otimes F_*^e(R)
\] for some $c \in R^{\circ}$ and for all $e \gg 0$. In general, let $N$ be a submodule of an $R$-module $M$. The Frobenius closure of $N$ in $M$, denoted by $N^F_M$, is the collection of elements in $M$ which lie in the kernel of the composition
\[
M\to M/N\xrightarrow{\id\otimes F^e} M/N\otimes_R F^e_*(R)
\]
for some $e\geq 0$. The tight closure of $N$ in $M$, denoted by $N^*_M$, is the collection of elements in $M$ which lie in the kernel of the composition
 \[
 M \to M/N \xrightarrow{\id\otimes F^e} M/N \otimes F_*^e(R) \xrightarrow{\id \otimes F_*^e(c)} M/N \otimes F_*^e(R)
 \]
for some $c \in R^{\circ}$ and for all $e \gg 0$. Both $N^F_M$ and $N^*_M$ are submodules of $M$ and there are containments 
 \[
 N\subseteq N^F_M\subseteq N^*_M\subseteq M.
 \]
Given an $R$-module $M$, $m\in M$, and $c\in R$ we will denote by $cm^{p^e}$ the element of $M\otimes F^e_*(R)$ which is mapped to by $m$ under the composition of maps
\[
M \xrightarrow{\id \otimes F^e}M\otimes F^e_*(R)\xrightarrow{\id\otimes F^e_*(c)} M\otimes F^e_*(R).
\]

\begin{remark}\label{direct system} Let $(R,\fm,k)$ be a local ring of prime characteristic $p>0$, $x_1, \ldots, x_t$ a sequence of elements in $\fm$, and $I = (x_1, \ldots, x_t)$.
 \begin{enumerate}
\item The $e$th-Frobenius action on $H^i_I(R)$ becomes a left $R$-module homomorphism $F^e: H^i_I(R) \to H^i_I(F^e_*(R))$ for all $i \ge 0$.    
 \item For each $e\in \NN$ there is an isomorphism of left $R$-modules $H^t_I(R) \otimes_R F^e_*(R) \cong H^t_I(F^e_*(R))$ (this is not true in general for lower degree local cohomology modules). If we identify the later with $H^t_I(R)$, then the map $H^t_I(R)  \xrightarrow{\mathrm{id} \otimes F^e} H^t_I(R) \otimes F^e_*(R)$ is the $e$th-Frobenius action on $H^t_I(R)$ . In particular, the nilpotent submodule of $H^t_I(R)$ is simply $0^{F}_{H^t_\fm(R)}$, the Frobenius closure of the $0$-submodule.

\item If $(R,\fm,k)$ is any local ring of prime characteristic $p>0$, $I\subseteq R$ an ideal, and $x_1,\ldots, x_t$ any choice of generators of $I$, then
\[
\lim_{\longrightarrow}\, \frac{(x_1^n, \ldots, x_t^n)^F}{(x_1^n, \ldots, x_t^n)} \cong 0^F_{H^t_I(R)}.
\]
See \cite[Proposition~3.3]{HQ18} for a proof (cf. Lemma~\ref{Direct limit of Frobenius closures lemma} below).

\item If $R$ is an excellent equidimensional local ring and $x_1,\ldots, x_t$ part of a system of parameters then
\[
\varinjlim_n\, \frac{(x_1^n, \ldots, x_t^n)^*}{(x_1^n, \ldots, x_t^n)} \cong 0^*_{H^t_I(R)}.
\]
See \cite[Propsition~3.3]{S94} for a proof (cf. Lemma~\ref{Direct limit of tight closures lemma} below).
\end{enumerate}
\end{remark}

An element $c\in R^{\circ}$ is called a test element if for all $R$-modules $N\subseteq M$ if $\eta\in N^*_M$ then for all $e\in \NN$ the element $\eta$ is an element of the kernel 
\[
M\to M/N\xrightarrow{\id\otimes F^e} M/N\otimes_R F^e_*(R)\xrightarrow{\id\otimes F^e_*(c)}M/N\otimes_R F^e_*(R).
\]
If $(R,\fm,k)$ is reduced and excellent then $R$ admits a test element. Moreover, we may assume $c\in R^\circ$ is a completely stable test element, i.e., $c$ is also serves as a test element for $\widehat{R}$.

\subsection{\texorpdfstring{$F$}{1}-nilpotent rings}

Let $(R,\fm,k)$ be a local ring of dimension $d$ and prime characteristic $p>0$. We say that a local cohomology module $H^i_I(R)$ is $F$-nilpotent if the Frobenius action on $H^i_I(R)$ is nilpotent. If $I=\fm$ is the maximal ideal then the modules $H^i_\fm(R)$ are Artinian and it follows that there is an $e\in \NN$ such that $\mbox{Span}_R\{\im(F^e)\}=\mbox{Span}_R\{\im(F^{e+1})\}=\cdots$. In particular, if $H^i_\fm(R)$ is $F$-nilpotent then $F^e:H^i_\fm(R)\to H^i_\fm(R)$ is the $0$-map  for all $e\gg 0$ (cf. \cite[Proposition~1.11]{HS77} and \cite[Proposition~4.4]{L97}). The ring $R$ is said to be $F$-nilpotent if $H^i_\fm(R)$ is $F$-nilpotent for all $i<d$ and $0^F_{H^d_\fm(R)}=0^*_{H^d_\fm(R)}$, i.e., the Frobenius action on $H^d_\fm(R)$ is nilpotent when restricted to the tight closure of the $0$-submodule, i.e., $0^{*}_{H^d_\fm(R)}=0^{F}_{H^d_\fm(R)}$.

The study of $F$-nilpotent rings is predominate in \cite{BB05} and \cite{ST17}. The authors of \cite{BB05} make an explict relation between the Lyubeznik numbers of a closed point $x$ in a variety $X$ defined over a seperably closed field $k$ of prime characteristic $p>0$ which is assumed to be $F$-nilpotent off $x$ and \'{e}tale cohomology groups of $X$ with coefficients in $\mathbb{F}_p$. The authors of \cite{ST17} propose a geometric interpretation of $F$-nilpotent singularities. Given isolated normal singularity $x\in X$ over $\mathbb{C}$, the authors of \cite{ST17} conjecture that a Hodge theoretic condition on the singularity $x$ is equivalent to the reduction mod $p$ of $X$ being $F$-nilpotent for almost all primes $p$. They verify their conjecture up to dimension $3$. 

This paper is concerned with algebraic and cohomological properties of $F$-nilpotent rings. We continue by discussing a well-known fact.

\begin{remark}\label{Frobenius action is not nilpotent} Suppose $x_1,\ldots, x_d$ is a system of parameters of $R$. The element $1+(x_1,\ldots, x_d)$ is a nonzero element of $H^d_\fm(R)\cong \varinjlim R/(x_1^N,\ldots, x_d^N)$ by the Monomial Conjecture.\footnote{The Monomial Conjecture is equivalent to the Direct Summand Conjecture. The reader can find proofs of these conjectures in the case $R$ contains a field in \cite{H75}. We also refer the reader to \cite{A16, B16} for proofs of the Direct Summand Conjecture for the case that $R$ does not contain a field.} The $e$th Frobenius action on $H^d_\fm(R)$ maps $1+(x_1,\ldots,x_d)\mapsto 1+(x_1^{p^e},\ldots, x_d^{p^e})$, which is also a nonzero element of $H^d_\fm(R)\cong \varinjlim R/(x_1^{p^eN},\ldots, x_d^{p^eN})$. Thus the Frobenius action on $H^d_\fm(R)$ is never nilpotent. 
\end{remark}

\begin{proposition}\label{Basic Proposition about F-nilpotent Rings}
Let $(R,\fm,k)$ be a local ring of prime characteristic $p>0$.
\begin{enumerate}
\item If $R$ has dimension $0$ then $R$ is $F$-nilpotent.
\item The ring $R$ is $F$-nilpotent if and only if $R/\sqrt{0}$ is $F$-nilpotent.
\item If $R$ is $F$-nilpotent then $R$ is equidimensional.
\item If $R$ is excellent then $R$ is $F$-nilpotent if and only if $\widehat{R}$ is $F$-nilpotent.
\end{enumerate}
\end{proposition}

\begin{proof}
If $R$ is of dimension $0$ then $H^0_\fm(R)=R$ and $0^*_{R}=\sqrt{0}$. 

For $(2)$ we let $\underline{x}=x_1,\ldots, x_d$ be a system of parameters for $R$. Then the Frobenius endomorphism induces a commutative diagram of short exact sequences of \v{C}ech complexes:
\[
\begin{CD}
0@>>>\check{C}^\bullet(\underline{x};\sqrt{0})@>>> \check{C}^\bullet(\underline{x};R) @>>> \check{C}^\bullet(\underline{x};R/\sqrt{0})@>>>0  \\
@.@VVF^{e}V @VVF^{e}V @VVF^{e}V@. \\
0@>>>\check{C}^\bullet(\underline{x};\sqrt{0})@>>>  \check{C}^\bullet(\underline{x};R) @>>>   \check{C}^\bullet(\underline{x};R/\sqrt{0})@>>>0. \\
\end{CD}
\]
Observe that $F^e:\check{C}^\bullet(\underline{x};\sqrt{0})\to \check{C}^\bullet(\underline{x};\sqrt{0})$ is the $0$-map for $e\gg 0$. It easily follows that for $0\leq i <d$ that $H^i_\fm(R)$ is $F$-nilpotent if and only if $H^i_\fm(R/\sqrt{0})$ is $F$-nilpotent. It remains to show $0^F_{H^d_\fm(R)}=0^*_{H^d_\fm(R)}$ if and only if $0^F_{H^d_\fm(R/\sqrt{0})}=0^*_{H^d_\fm(R/\sqrt{0})}$. Suppose first that $0^F_{H^d_\fm(R)}=0^*_{H^d_\fm(R)}$ and let $\eta\in 0^*_{H^d_\fm(R/\sqrt{0})}$. Denote by $\iota$ and $\pi$ the following maps of local cohomology modules:
\[
H^d_\fm(\sqrt{0})\xrightarrow{\iota}H^d_\fm(R)\xrightarrow{\pi}H^d_\fm(R/\sqrt{0}).
\]
The map $\pi$ is onto and so there exists $\tilde{\eta}\in H^d_\fm(R)$ so that $\pi(\tilde{\eta})=\eta$. The commutative diagram of \v{C}ech complexes induces the following commutative diagram of local cohomology modules:
\[
\begin{CD}
H^d_\fm(\sqrt{0})@>\iota>>H^d_\fm(R) @>\pi>> H^d_\fm(R/\sqrt{0})@>>>0  \\
@VVF^{e}V @VVF^{e}V @VVF^{e}V@. \\
H^d_\fm(\sqrt{0})@>\iota>> H^d_\fm(R) @>\pi>>   H^d_\fm(R/\sqrt{0})@>>>0. \\
\end{CD}
\]
It follows that there is a $c\in R^{\circ}$ so that $c\tilde{\eta}^{p^e}\in \im (\iota)$ for all $e\gg0$. However, if $e_0\in \NN$ is chosen such that $\sqrt{0}^{[p^{e_0}]}=0$ then $F^{e_0}(\gamma)=0$ for all $\gamma\in H^d_\fm(\sqrt{0})$. A simple diagram chase then shows $c^{p^{e_0}}\tilde{\eta}^{p^{e+e_0}}=0$ for all $e\in \NN$, i.e., $\tilde{\eta}\in 0^*_{H^d_\fm(R)}$. But we are assuming $R$ is $F$-nilpotent, hence for $e\gg 0$ we have $\tilde{\eta}^{p^e}=0$ and another diagram chase shows $\eta\in 0^F_{H^d_\fm(R/\sqrt{0})}.$ We leave it to the reader to run a similar argument proving if $0^F_{H^d_\fm(R/\sqrt{0})}=0^*_{H^d_\fm(R/\sqrt{0})}$ then $0^F_{H^d_\fm(R)}=0^*_{H^d_\fm(R)}$.

To prove $(3)$ we may assume $R$ is reduced. Let $\{P_1,\cdots, P_\ell\}$ be the minimal primes of $R$ and suppose for a contradiction that $\dim(R/P_1)=i<\dim(R)$. Let $I=P_2\cap \cdots \cap P_\ell$. Consider the short exact sequence
\[
0\to R\to \frac{R}{P_1}\oplus \frac{R}{I}\to \frac{R}{(P_1+I)}\to 0.
\]
It follows that there is an onto map $H^i_\fm(R)\to H^i_\fm(R/P_1)$ of local cohomology modules since $\dim (R/(P_1+I))<i$. A straightforward diagram chase proves $H^i_\fm(R/P_1)$ is $F$-nilpotent, a contradiction since the top local cohomology module of a local ring cannot be $F$-nilpotent.

For $(4)$ we begin by recalling $H^i_\fm(R)\cong H^i_\fm(R)\otimes_R \widehat{R}\cong H^i_{\fm \widehat{R}}(\widehat{R})$ for all $i\in \NN$. Suppose first $\widehat{R}$ is $F$-nilpotent. Clearly $H^i_\fm(R)$ are $F$-nilpotent for all $i<\dim(R)$. It remains to show $0^{*_R}_{H^d_{\fm(R)}}=0^{F_R}_{H^d_{\fm}(R)}$. But this is also clear since $R^\circ\subseteq \widehat{R}^{\circ}$, hence $0^{*_R}_{H^d_\fm(R)}\subseteq 0^{*_{\widehat{R}}}_{H^d_\fm(R)}=0^{F_{\widehat{R}}}_{H^d_\fm(R)}$ and therefore a large iterate of the Frobenius action on $H^d_\fm(R)$ maps $\eta$ to $0$.

Conversely, suppose that $R$ is $F$-nilpotent. To ease notation, given a ring $S$ we write $S_{red}$ to denote $S/\sqrt{0}$. To show $\widehat{R}$ is $F$-nilpotent it is enough to show $\widehat{R}_{red}$ is $F$-nilpotent. Observe $\widehat{R}_{red}\cong (\widehat{R_{red}})_{red}$ and so we instead prove $\widehat{R_{red}}$ is $F$-nilpotent. The ring $R_{red}$ is $F$-nilpotent by $(2)$. Therefore we may assume $R$ is an excellent reduced ring. All lower local cohomology modules of $\widehat{R}$ are $F$-nilpotent by assumption and it remains to show $0^{*_{\widehat{R}}}_{H^d_\fm(R)}=0^{F_{\widehat{R}}}_{H^d_\fm(R)}$. Let $\eta\in 0^{*_{\widehat{R}}}_{H^d_\fm(R)}$. The ring $R$ admits a complete stable test element and therefore $\eta$ is an element of $0^{*_R}_{H^d_\fm(R)}$ as well. By assumption, a large enough iterate of the Frobenius action on $H^d_\fm(R)$ must map $\eta$ to $0$. 
\end{proof}

\begin{remark} The proof of (3) of Proposition~\ref{Basic Proposition about F-nilpotent Rings} did not fully use the hypothesis that $R$ is $F$-nilpotent. The proof only required that the Frobenius actions on $H^{i}_\fm(R)$ were nilpotent for $i<\dim(R)$. Using the language of \cite[Section~4]{Lyubeznik2006}, if the $F$-depth of a local ring $R$ is equal to the dimension of $R$, then $R$ is necessarily an equidimensional local ring.
\end{remark}

\section{On the complexity of Frobenius powers of ideals}\label{LC section}
Let $(R,\fm,k)$ be a local ring of prime characteristic $p>0$. An ideal $I\subseteq R$ is said to satisfy condition (LC) if there is an integer $C$ such that for each $e\in \NN$
\[
\fm^{Cp^e}\cdot H^0_\fm(R/I^{[p^e]})=0.
\]
Showing every ideal satisfies condition (LC) would have an important application. If $R$ is assumed to be weakly $F$-regular, i.e., every ideal is tightly closed, and every ideal of $R$ satisfies condition (LC), then every localization of $R$ would remain weakly $F$-regular. A ring whose localizations are weakly $F$-regular is called $F$-regular. We refer the reader to the discussion following \cite[Proposition~4.16]{HH90} or the discussion following \cite[Corollary~3.2]{H00} for further details on why every weakly $F$-regular ring satisfying condition (LC) is $F$-regular. 


\begin{conjecture}[(LC) Conjecture]\label{LC conj} Let $(R, \fm, k)$ be a local (graded) ring of prime characteristic $p>0$ and $I$ a (homogeneous) ideal of $R$. Then there exists a positive integer $C$ such that 
$$\fm^{Cp^e}\cdot H^0_{\fm}(R/I^{[p^e]}) = 0$$
for all $e \ge 0$.
\end{conjecture}

There has been limited progress towards a proof of  Conjecture~\ref{LC conj}. See \cite[Corollary~3.2]{H00} and \cite[Theorem~1]{V00} for proofs that a homogeneous ideal $I$ of a equidimensional graded ring such that $\dim (R/I)=1$ satisfies the (LC) condition. We also refer the reader to \cite[Theorem~20]{K98} for a similar result which should be compared to Theorem~\ref{general Katzman}, the main result of this section, found below. Specifically, Theorem~\ref{general Katzman} shows that an ideal generated by a filter regular sequence satisfies condition (LC). We begin with a pair of lemmas.

\begin{lemma}\label{ann of lc} Let $(R, \fm,k)$ be a local ring, of arbitrary characteristic, and $t$ a non-negative integer.  For each $j \ge 0$ let $\fa_j(R) = \mathrm{Ann}(H^j_{\fm}(R))$ and $\fa = \fa_0(R) \cdots \fa_t(R)$. If $x_1, \ldots, x_t$ a filter regular sequence then
\[
\fa^{2^t}\cdot H^0_{\fm}(R/(x_1, \ldots, x_t)) = 0
\]
\end{lemma}

\begin{proof} There is nothing to show if $t=0$. If $t>0$ consider the long exact sequence of local cohomology
\[
\cdots \to H^i_{\fm}(R) \to H^i_{\fm}(R/(x_1)) \to H^{i+1}_{\fm}(R) \to \cdots
\]
induced from the short exact sequence
\[
0\rightarrow R/(0:x_1) \xrightarrow{x_1}R\rightarrow R/(x_1)\rightarrow 0.
\] 
Then for each $i \le t-1$ we have
\[
\fa_i(R/(x_1))=\mathrm{Ann}~ (H^i_{\fm}(R/(x_1))) \supseteq \fa_i(R) \fa_{i+1}(R) \supseteq \fa^2.
\]
The assertion follows inductively as for each $i\leq t-(j+1)$ there will be short exact sequence
\[
\cdots \to H^i_{\fm}(R/(x_1,\ldots, x_j)) \to H^i_{\fm}(R/(x_1, \dots, x_j,x_{j+1})) \to H^{i+1}_{\fm}(R/(x_1,\ldots, x_j)) \to \cdots. \qedhere
\]
\end{proof}

For the next lemma we first recall the notion of an attached prime of an Artinian module and some of their basic properties due to Macdonald, \cite{M73}. Let $(R,\fm,k)$ be a local ring and $M$ a nonzero Artinian $R$-module. We say that $M$ is secondary if for each $x\in R$ the multiplication map $M\xrightarrow{x}M$ is either onto or nilpotent. If $M$ is secondary then $P=\sqrt{\Ann_R (M)}$ is a prime ideal and we call $M$ a $P$-secondary module. A secondary representation of an Artinian $R$-module $M$ is a decomposition $M=N_1+\cdots +N_\ell$ such that each $N_i$ is secondary. The chosen secondary representation of $M$ is called minimal if $N_i\not\subseteq N_j$ for all $i\neq j$ and $\sqrt{\Ann_R(N_i)}\neq \sqrt{\Ann_R(N_j)}$ for all $i\neq j$. We say that the prime ideals $\sqrt{\Ann_R(N_i)}$ are attached primes of $M$. Minimal secondary representations of Artinian modules always exist, are not unique, but the list of attached primes associated with a minimal secondary representation is unique. The set of such prime ideals is denoted by $\Att_R(M)$.

\begin{lemma}\label{m primary} Let $(R, \fm,k)$ be a local ring which is the image of a Cohen-Macaulay local ring (e.g. $R$ is excellent
\footnote{Kawasaki proved every excellent local ring is the homomorphic image of a Cohen-Macaulay local ring, see \cite[Corollary~1.2]{K02}.}), and $t$ a non-negative integer.  For each $j \ge 0$ let $\fa_j(R) = \mathrm{Ann}(H^j_{\fm}(R))$ and $\fa = \fa_0(R) \cdots \fa_t(R)$. If $x_1, \ldots, x_t$ is a filter regular sequence then $\fa + (x_1, \ldots, x_t)$ is $\fm$-primary.
\end{lemma}
\begin{proof} It is enough to show $\fa_j(R) + (x_1, \ldots, x_t)$ is $\fm$-primary for each $j \le t$. Without loss of generality we assume $H^j_\fm(R)\not=0$ and $t<d$. Suppose there is a prime ideal $P \neq \fm$ containing $\fa_j(R) + (x_1, \ldots, x_t)$. The module $H^j_\fm(R)$ is Artinianian and $P$ contains the annihilator of $H^j_\fm(R)$ by assumption. Therefore by \cite[7.2.11 (ii)]{BS13} there exists $Q \in \mathrm{Att}_R(H^j_{\fm}(R))$ such that $Q \subseteq P$. By \cite[Theorem~1.1]{NQ14} we have $Q R_P \in \mathrm{Att}_{R_P} H^{j - \dim R/P}_{PR_P}(R_P)$. Thus $ H^{j - \dim R/P}_{PR_P}(R_P) \neq 0$. However, $x_1, \ldots, x_t$ is a filter regular sequence in $P \neq \fm$. So $x_1, \ldots, x_t$ becomes a regular sequence after localization at $P$ by Lemma~\ref{filter}. So we also have shown $ H^j_{PR_P}(R_P) = 0$ for all $j<t$, a contradiction.
\end{proof}


We are now ready to prove the main result of this section.

\begin{theorem}\label{general Katzman} Let $(R, \fm, k)$ be a local ring of characteristic $p>0$. Let $x_1, \ldots, x_t$ be a filter regular sequence of $R$ and $I = (x_1, \ldots, x_t)$. Then there exists a positive integer $C$ such that 
$$\fm^{Cp^e}\cdot H^0_{\fm}(R/I^{[p^e]}) = 0$$
for all $e \ge 0$.
\end{theorem}
\begin{proof} The property that $x_1,\ldots,x_t$ is a filter regular sequence is unaffected by passing to the completion of $R$, see (4) of Lemma~\ref{filter}. In particular, by passing to the completion, we may assume $R$ is the homomorphic image of a Cohen-Macaulay local ring. We can further assume that $t>0$. For each $j \ge 0$ let $\fa_j(R) = \mathrm{Ann}(H^j_{\fm}(R))$ and $\fa = \fa_0(R) \cdots \fa_t(R)$. The sequence $x_1^{p^e}, \ldots, x_t^{p^e}$ is a filter regular sequence for each $e\geq 0$ by Lemma~\ref{filter}. Recall that $\fa^{2^t}\cdot  H^0_{\fm}(R/I^{[p^e]}) = 0$ for all $e \ge 0$  by Lemma~\ref{ann of lc}. Thus $\fa^{2^t} + I^{[p^e]} \subseteq \mathrm{Ann}~H^0_{\fm}(R/I^{[p^e]})$ for all $e \ge 0$. By Lemma~\ref{m primary} the ideal $\fa^{2^t} + I$ is $\fm$-primary. So we can choose $C_1$ such that $\fm^{C_1} \subseteq \fa^{2^t} + I$, and so
\[
(\fm^{[p^e]})^{C_1} = (\fm^{C_1})^{[p^e]}\subseteq (\fa^{2^t} + I)^{[p^e]} \subseteq \fa^{2^t} + I^{[p^e]}
\]
for all $e \ge 0$. Then we can choose a suitable multiple $C$ of $C_1$ such that 
\[
\fm^{Cp^e}  \subseteq \fa^{2^t} + I^{[p^e]} \subseteq \mathrm{Ann}~H^0_{\fm}(R/I^{[p^e]})
\]
for all $e \ge 0$. 
\end{proof}

\section{Nilpotence of Frobenius action on local cohomology and deformation}\label{Deformation Section}

Continue to let $(R,\fm,k)$ be a local ring of prime characteristic $p>0$. The first theorem of the section is a deformation type result for 
$F$-nilpotent rings. More specifically, Theorem~\ref{Deformation type statement} provides a necessary and sufficient criteria to determine if the Frobenius endomorphism on $H^t_\fm(R)$ is nilpotent by examining the behavior of the Frobenius endomorphism on $H^i_\fm(R/(x))$ for all $i\leq t-1$ for general choice of parameter element $x\in R$. The example following the proof of Theorem~\ref{Deformation type statement} indicates that the result is the best possible towards deforming $F$-nilpotent singularities. The reader should observe that the Frobenius action on the $0$th local cohomology module will always be nilpotent whenever $\dim(R)>0$. In particular, if $R$ is of Krull dimension at least $2$ with parameter element $x$, then the Frobenius action on $H^0_\fm(R/(x))$ will always be nilpotent and  cannot detect if the Frobenius action on $H^1_\fm(R)$ is nilpotent. To describe our deformation type result we will need to discuss the notion of being $F$-nilpotent relative to $R$.

 Let $I\subseteq K$ be ideals of $R$. The Frobenius endomorphism $F:R/I\to R/I$ can be factored as composition of two natural maps: $R/ I  \to R/I^{[p]} \to R/I$, where the first map is obtained by base change along the Frobenius endomorphism $F:R\to R$ and the second map is the natural projection map. We denote the first map by $F_R$: $F_R(a + I) = a^p + I^{[p]}$ for all $a \in R$.  We say that a local cohomology module $H^i_K(R/I)$ is F-nilpotent with respect to $R$ if for every $\eta\in H^i_K(R/I)$ there is an $e \gg 0$ such that $F^e_R(\eta)=0$ in $H^i_K(R/I^{[p^e]})$ where $F^e_R:H^i_K(R/I)\rightarrow H^i_K(R/I^{[p^e]})$ is the natural map induced by the $e$th Frobenius morphism $F^e_R: R/I\rightarrow R/I^{[p^e]}$.

\begin{remark}
Observe that $H^0_K(R/I)\cong (I:K^\infty)/I=\bigcup_{n=1}^\infty (I:K^n) /I$. Therefore $H^0_K(R/I)$ is $F$-nilpotent with respect to $R$ if and only if for all $e\gg 0$, $(I:K^\infty)^{[p^e]}\subseteq I^{[p^e]}$. This is equivalent to the requirement that $ (I:K^\infty)\subseteq I^F$. It should also be noted that $H^0_K(R/I)$ is always $F$-nilpotent whenever $K\not\subseteq P$ for all $P\in \min(I)$.
\end{remark}

\begin{theorem}\label{Deformation type statement} Let $(R,\fm, k)$ be a local ring of prime characteristic $p>0$ and of dimension $d>0$. For every integer $t$ the following are equivalent:

\begin{enumerate}

\item $H^i_\fm(R)$ is $F$-nilpotent for all $ i\leq t$.

\item For every filter regular element $x$, $H^{i}_\fm(R/(x))$ is F-nilpotent with respect to $R$ for every $i \leq t-1$.

\item There exists a filter regular element $x$ such that for all $n\geq 1$, $H^{i}_\fm(R/(x^n))$ is F-nilpotent with respect to $R$ for all $i \leq t-1$.

\item There exists a filter regular element $x$ such that for all $e\geq 0$, $H^{i}_\fm(R/(x^{p^e}))$ is F-nilpotent with respect to $R$ for every $i\leq t-1$.

\end{enumerate}
\end{theorem}
\begin{proof} $(1) \Rightarrow (2)$, suppose that $H^i_\fm(R)$ is $F$-nilpotent for each $i \leq t$ and let $x\in \fm$ be a filter regular element. The multiplication map $R \overset{x}{\to} R$ induces the short exact sequence 
$$0\rightarrow R/(0:x) \xrightarrow{x}R\rightarrow R/(x)\rightarrow 0.$$ 
Let $i\leq t-1$ and pick $\eta\in H^i_\fm(R/(x))$. Let $\delta$ be the induced connecting homomorphism of local cohomology modules 
$$H^i_\fm(R/xR)\rightarrow H^{i+1}_\fm(R/(0:x)) \cong H^{i+1}_\fm(R).$$ 
Choose $e \gg 0$ such that $F^e(\delta(\eta))=0$ where $F^e:H^{i+1}_\fm(R)\rightarrow H^{i+1}_\fm(R)$ is the $e$-th Frobenius action on $H^{i+1}_\fm(R)$. Consider the following commutative diagram whose rows are exact:
\[
\begin{CD}
H^i_{\fm}(R) @>>> H^i_{\fm}(R/(x)) @>\delta>> H^{i+1}_{\fm}(R)  \\
@VVF^{e}V @VVF^{e}_RV @VVF^{e}V \\
H^i_{\fm}(R)  @>\beta>> H^i_{\fm}(R/(x^{p^e})) @>>>  H^{i+1}_{\fm}(R). \\
\end{CD}
\]
Therefore $F^e_R(\eta)=\beta(\eta')$ for some $\eta'\in H^{i}_\fm(R)$. Hence there is $e' \gg 0$ such that $F^{e'}(\eta')=0$. Then examination of the following commutative diagram, whose rows are exact, shows that $F^{e+e'}_R(\eta)=0$ in $H^i_\fm(R/(x^{p^{e+e'}}))$.
\[
\begin{CD}
H^i_{\fm}(R) @>>> H^i_{\fm}(R/(x)) @>\delta>> H^{i+1}_{\fm}(R)  \\
@VVF^{e}V @VVF^{e}_RV @VVF^{e}V \\
H^i_{\fm}(R)  @>\beta>> H^i_{\fm}(R/(x^{p^e})) @>>>  H^{i+1}_{\fm}(R).\\
@VVF^{e'}V @VVF^{e'}_RV  \\
H^i_{\fm}(R)  @>>> H^i_{\fm}(R/(x^{p^{e+e'}})). \\
\end{CD}
\]
This completes the implication $(1)\Rightarrow (2)$. Clearly, $(2)\Rightarrow (3)\Rightarrow (4)$. We now show that $(4)\Rightarrow (1)$. Let $x$ be a filter regular element on $R$ such that $H^i_\fm(R/(x^{p^e}))$ is $F$-nilpotent with respect to $R$ for every $e$ and every $i \leq t-1$. We can assume that $t \ge 1$. Since $H^0_{\fm}(R)$ is always $F$-nilpotent we need only to show $H^i_{\fm}(R)$ is $F$-nilpotent for all $1 \le i \le t$. Suppose that $\eta\in H^i_\fm(R)$. After replacing $x$ by $x^{p^e}$ for some $e \gg 0$, we may assume that $x\eta=0$. We can then take $e$ large, and then chase the following diagram to conclude that $F^{e}(\eta)=0$.
$$
\begin{CD}
H^{i-1}_{\fm}(R/(x)) @>>> H^i_{\fm}(R) @>x>> H^{i}_{\fm}(R)  \\
@VVF^{e}_RV @VVF^{e}V @VVF^{e}V \\
H^i_{\fm}(R/(x^{p^e}))  @>>> H^i_{\fm}(R) @>>>  H^{i}_{\fm}(R). \\
\end{CD}
$$
\end{proof}

\begin{example}\label{non-deformation example} Let $R$ be the localization of $\FF_p[T_1,T_2,T_3]/(T_1^2T_2,T_1^2T_3)$ at the ideal $(T_1, T_2, T_3)$. Let $t_1, t_2, t_3$ represent the classes of $T_1,T_2,T_3$ in $R$ respectively. Observe that $R$  is local ring of prime characteristic $p$ of dimension $2$, depth $1$, and $x=t_1+t_2$ is an $R$-regular element. Moreover, $H^0_\fm(R/(x))$ is $F$-nilpotent with respect to $R$ but $H^0_\fm(R/(x^p))$ is not $F$-nilpotent with respect to $R$. To see this observe that 
\begin{align*}
H^0_\fm(R/(x))\cong& \frac{(x,t_1^2)}{(x)}\mbox{ and }\\
H^0_\fm(R/(x^{p^e}))\cong& \frac{(t_1^2,t_2^{p^e})}{(x^{p^e})}\mbox{ for all }e\geq 1.
\end{align*}
It is then simple  to see that every element of $H^0_\fm(R/(x))$ is mapped to $0$ under the Frobenius map $H^0_\fm(R/(x))\to H^0_\fm(R/(x^p))$. However, the element $t_1^2$ of $H^0_\fm(R/(x^p))$ cannot be mapped to $0$ in $H^0_\fm(R/(x^{p^{e+1}}))$ under Frobenius for all $e\geq 1$ by degree considerations.
\end{example}

We remind the reader that if $I$ is an ideal generated by a filter regular sequence and if $x$ is a filter regular element on $R/I$, then $x$ is a filter regular element on $R/I^{[p^e]}$ for all $e\geq 0$ and $I$ satisfies the (LC) condition by Theorem~\ref{general Katzman}. 

\begin{theorem}\label{general Deformation} Let $(R,\fm, k)$ be a local ring of prime characteristic $p>0$ and of dimension $d$. Let $I$ be an ideal of $R$ and suppose that $x\in R$ is a filter regular element on $R/I^{[p^e]}$ for all $e \ge 0$.
\begin{enumerate}
\item If $H^i_\fm(R/I^{[p^e]})$ is $F$-nilpotent with respect to $R$ for all $ i\leq t$ and for all $e \ge 0$ then $H^{i}_\fm(R/(I^{[p^e]}, x^{p^{e'}}))$ is $F$-nilpotent with respect to $R$ for all $i \leq t-1$ and for all $e,e' \ge 0$.
\item Conversely, if $I$ satisfies the (LC) condition and $H^{i}_\fm(R/(I^{[p^e]}, x^{p^{e'}}))$ is $F$-nilpotent with respect to $R$ for all $i \leq t-1$ and for all $e, e' \ge 0$, then $H^i_\fm(R/I^{[p^e]})$ is $F$-nilpotent with respect to $R$ for all $ i\leq t$ and for all $e \ge 0$.
\end{enumerate}
\end{theorem}

\begin{proof} Replacing $I^{[p^e]}$ by $I$ and $x^{p^{e'}}$ by $x$, it is enough to show that $H^{i}_\fm(R/(I, x))$ is $F$-nilpotent with respect to $R$ for all $i \leq t-1$. The multiplication map $R/I \overset{x}{\to} R/I$ induces the short exact sequence 
$$0\rightarrow R/(I:x) \xrightarrow{x}R/I \rightarrow R/(I,x)\rightarrow 0.$$ 
Since $x$ is a filter regular element of $R/I$, we have $(I:x)/I$ has finite length. So $H^i_{\fm}(R/(I:x)) \cong H^i_{\fm}(R/I)$ for all $i \ge 1$. Thus we have the induced long exact sequence of local cohomology
$$\cdots \to H^i_{\fm}(R/I) \to H^i_{\fm}(R/(I,x)) \overset{\delta}{\to} H^{i+1}_{\fm}(R/I) \xrightarrow{x} H^{i+1}_{\fm}(R/I) \to \cdots.$$
There is a commutative diagram
\[
\begin{CD}
0 @>>> R/(I:x) @>x>> R/I @>>> R/(I,x) @>>> 0   \\
@. @VV(F^{e}_R)'V @VVF^{e}_RV @VVF^{e}_RV\\
0 @>>> R/(I^{[p^e]}:x^{p^e})  @>x^{p^e}>> R/I^{[p^e]} @>>> R/(I,x)^{[p^e]} @>>> 0, \\
\end{CD}
\]
where the most left vertical map is the composition 
\[
R/(I:x) \overset{F^e_R}{ \longrightarrow } R/(I:x)^{[p^e]} \twoheadrightarrow R/(I^{[p^e]}:x^{p^e}).
\]
Moreover since $x$ and $x^{p^e}$ are filter regular elements of $R/I$ and $R/I^{[p^e]}$ respectively, the derived maps of local cohomology modules in positive degrees can be identified with the Frobenius maps
\[
F^e_R: H^i_{\fm}(R/I) \longrightarrow H^i_{\fm}(R/I^{[p^e]}).
\]
Therefore we have the following commutative diagram whose rows are exact for any $i \ge 0$
\[
\begin{CD}
H^i_{\fm}(R/I) @>>> H^i_{\fm}(R/(I,x)) @>\delta>> H^{i+1}_{\fm}(R/I)  \\
@VVF^{e}_RV @VVF^{e}_RV @VVF^{e}_RV \\
H^i_{\fm}(R/I^{[p^e]})  @>\beta>> H^i_{\fm}(R/(I,x)^{[p^e]}) @>>>  H^{i+1}_{\fm}(R/I^{[p^e]}). \\
\end{CD}
\]
Let $i\leq t-1$ and pick $\eta\in H^i_\fm(R/(I,x))$. Choose $e \gg 0$ such that $F^e_R(\delta(\eta))=0$. Therefore $F^e_R(\eta)=\beta(\eta')$ for some $\eta'\in H^{j}_\fm(R)$. Hence there is $e' \gg 0$ such that $F^{e'}(\eta')=0$. Then the following commutative diagram, whose rows are exact, shows that $F^{e+e'}_R(\eta)=0$ in $H^i_\fm(R/((I,x)^{[p^{e+e']}})$, i.e., $H^i_\fm(R/(I,x))$ is $F$-nilpotent with respect to $R$.
\[
\begin{CD}
H^i_{\fm}(R/I) @>>> H^i_{\fm}(R/(I,x)) @>\delta>> H^{i+1}_{\fm}(R/I)  \\
@VVF^{e}_RV @VVF^{e}_RV @VVF^{e}_RV \\
H^i_{\fm}(R/I^{[p^e]})  @>\beta>> H^i_{\fm}(R/(I,x)^{[p^e]}) @>>>  H^{i+1}_{\fm}(R/I^{[p^e]}). \\
@VVF^{e'}_RV @VVF^{e'}_RV  \\
H^i_{\fm}(R/I^{[p^{e+e'}]})  @>>> H^i_{\fm}(R/(I,x)^{[p^{e+e'}]}). \\
\end{CD}
\]
Conversely, suppose $I$ satisfies Conjecture~\ref{LC conj} and $H^{i}_\fm(R/(I^{[p^e]}, x^{p^{e'}}))$ is $F$-nilpotent with respect to $R$ for all $i \leq t-1$ and for all $e, e' \ge 0$. By replacing $I^{[p^e]}$ by $I$, it is enough to show that $H^i_{\fm}(R/I)$ is $F$-nilpotent for all $i \le t$. We are assuming there exists $C\in \NN$ such that $\fm^{C p^e} H^0_{\fm}(R/I^{[p^e]}) = 0$ for all $e \ge 0$, i.e., $H^0_\fm(R/I^{[p^e]})=(I^{[p^e]}:_R \fm^{Cp^e})/I^{[p^e]}$ for all $e \ge 0$. Choose $e_0$ so that $p^{e_0} \ge C$ and replace $x$ by $x^{p^{e_0}}$. Then $H^0_{\fm}(R/(I^{[p^e]}:x^{p^e})) \cong 0$ by Lemma~\ref{0th local cohomology lemma} for all $e \ge 0$. Therefore, the commutative diagram 
\[
\begin{CD}
0 @>>> R/(I:x) @>x>> R/I @>>> R/(I,x) @>>> 0   \\
@. @VV(F^{e}_R)'V @VVF^{e}_RV @VVF^{e}_RV\\
0 @>>> R/(I^{[p^e]}:x^{p^e})  @>x^{p^e}>> R/I^{[p^e]} @>>> R/(I,x)^{[p^e]} @>>> 0, \\
\end{CD}
\]
induces the following commutative diagram whose rows are exact for all $e \ge 0$:
\[
\begin{CD}
0 @>>> H^0_{\fm}(R/I) @>>> H^0_{\fm}(R/(I,x))  \\
@. @VVF^{e}_RV @VVF^{e}_RV \\
 0 @>>> H^0_{\fm}(R/I^{[p^e]}) @>>> H^0_{\fm}(R/(I,x)^{[p^e]}). \\
\end{CD}
\]
It easily follows that the assumption $H^0_{\fm}(R/(I,x))$ is $F$-nilpotent with respect to $R$ implies $H^0_{\fm}(R/I)$ is $F$-nilpotent with respect to $R$. We next show that $H^i_{\fm}(R/I)$ is $F$-nilpotent with respect to $R$ for each $1 \le i \le t$. Let $\eta\in H^i_\fm(R/I)$. After replacing $x$ by $x^{p^e}$ for some $e \gg 0$, we may assume that $x\eta=0$. We can then take $e$ large, and then chase the following diagram to conclude that $F^{e}(\eta)=0$.
\[
\begin{CD}
H^{i-1}_{\fm}(R/(I, x)) @>>> H^i_{\fm}(R/I) @>x>> H^{i}_{\fm}(R/I)  \\
@VVF^{e}_RV @VVF_R^{e}V @VVF_R^{e}V \\
H^i_{\fm}(R/(I,x)^{[p^e]})  @>>> H^i_{\fm}(R/I^{[p^e]}) @>>>  H^{i}_{\fm}(R/I^{[p^e]}). \\
\end{CD}
\]
\end{proof}

\begin{remark}\label{filter ideal}
Let $(R,\fm,k)$ be a local ring of prime characteristic $p>0$ and suppose $I\subseteq R$ is an ideal satisfying condition
\[
\mathrm{Ass}_R R/I^{[p^e]} \subseteq \mathrm{Ass}_R R/I \cup \{\fm \}.
\]
Then every filter regular element of $R/I$ is also a filter regular element of $R/I^{[p^e]}$ for all $e \ge 0$. Ideals generated by a filter regular sequence will always satisfy the above condition, Remark \ref{filter}.
\end{remark}

\begin{theorem}\label{characterization nilpotent 1} Let $(R,\fm, k)$ be a local ring of prime characteristic $p>0$ and of dimension $d$. Let $t<d$. Then the following are equivalent:

\begin{enumerate}

\item $H^i_\fm(R)$ is F-nilpotent for all $ i\leq t$.

\item For every filter regular sequence $x_1, \ldots, x_t$ we have $(x_1, \ldots, x_t) : \fm^{\infty} \subseteq (x_1, \ldots, x_t)^F$.

\item There exists a filter regular sequence $x_1, \ldots, x_t$ such that 
$$(x_1^{n_1}, \ldots, x_t^{n_t}) : \fm^{\infty} \subseteq (x_1^{n_1}, \ldots, x_t^{n_t})^F$$ 
for all $n_1, \ldots, n_t \ge 1$.

\item There exists a filter regular sequence $x_1, \ldots, x_t$ such that 
$$(x_1^{p^{e_1}}, \ldots, x_t^{p^{e_t}}) : \fm^{\infty} \subseteq (x_1^{p^{e_1}}, \ldots, x_t^{p^{e_t}})^F$$ 
for all $e_1, \ldots, e_t \ge 0$.

\end{enumerate}
\end{theorem}
\begin{proof} For $(1) \Rightarrow (2)$, let $x_1, \ldots, x_i$, $i \le t$ be a filter regular sequence. By Theorem \ref{general Deformation} and Remark \ref{filter ideal} we have $H^j_m(R/(x_1, ..., x_i))$ is F-nilpotent with respect to R for all $i+j \le t$. In particular, we have $H^0_{\fm}(R/(x_1, \ldots, x_t))$ is $F$-nilpotent with respect to $R$. This condition is equivalent to the claim that
 \[
 (x_1, \ldots, x_t) : \fm^{\infty} \subseteq (x_1, \ldots, x_t)^F.
 \]
 The implications $(2) \Rightarrow (3) \Rightarrow (4)$ are trivial.
 
Our proof of $(4)\Rightarrow (1)$ begins with the observation that all hypotheses and desired conclusions are unaffected by passing to the completion of $R$, see Lemma~\ref{filter} and Proposition~\ref{Basic Proposition about F-nilpotent Rings}. Condition $(4)$ is equivalent to $H^0_{\fm}(R/(x_1^{p^{e_1}}, \ldots, x_t^{p^{e_t}}))$ being $F$-nilpotent with respect to $R$ for all $e_1, \ldots, e_t \ge 0$. We next show that $H^i_{\fm}(R/(x_1^{p^{e_1}}, \ldots, x_{t-1}^{p^{e_{t-1}}}))$ is $F$-nilpotent with respect to $R$ for all $i \le 1$ and for all $e_1, \ldots, e_{t-1} \ge 0$. For each $e_1, \ldots, e_{t-1} \ge 0$ we denote $I_{e_1,\ldots, e_{t-1}} = (x_1^{p^{e_1}}, \ldots, x_{t-1}^{p^{e_{t-1}}})$. We have $H^0_{\fm}(R/(I_{e_1,\ldots, e_{t-1}}^{[p^e]}, x_t^{p^{e'}}))$ is $F$-nilpotent with respect to $R$ for all $e, e' \ge 0$. On the other hand $I_{e_1,\ldots, e_{t-1}}$ is generated by a filter regular sequence, so its satisfies Conjecture \ref{LC conj} by Theorem \ref{general Katzman}. By Theorem \ref{general Deformation} we have $H^i_{\fm}(R/I_{e_1,\ldots, e_{t-1}})$ is $F$-nilpotent with respect to $R$ for all $i \le 1$. We continue this progress with the same method and we have $(1)$.\end{proof}


We next use the Nagel-Schenzel isomorphism, Lemma~\ref{nagel-shenzel}, to provide another equivalent characterization for Frobenius actions to be nilpotent on local cohomology modules.

\begin{proposition}\label{frs imply nilpotent} Let $(R,\fm,k)$ be local ring of prime characteristic $p$. Let $x_1,...,x_t$ be a filter regular sequence, $I=(x_1,...,x_t)$, and suppose that for each $e>0$ that $(I^{[p^e]}:\fm^{\infty})\subseteq (I^{[p^e]})^F$, then $H^t_\fm(R)$ is F-nilpotent.

\end{proposition}

\begin{proof} By Lemma~\ref{nagel-shenzel} there is an isomorphism $H^t_\fm(R)\cong H^0_\fm(H^t_{I}(R))$.
Moreover, we may realize $H^t_I(R)$ as the direct limit system
\[
H^t_I(R)\cong\varinjlim(\cdots\rightarrow R/I^{[p^e]} \rightarrow R/I^{[p^{e+1}]}\rightarrow \cdots)
\] 
where $R/I^{[p^e]} \rightarrow R/I^{[p^{e+1}]}$ is multiplication by $(x_1\cdots x_t)^{p^{e+1}-p^e}$. Therefore
\begin{eqnarray*}
H^t_\fm(R) \cong H^0_\fm(H^t_I(R))  &\cong& \varinjlim(\cdots\rightarrow H^0_\fm(R/I^{[p^e]}) \rightarrow H^0_\fm(R/I^{[p^{e+1}]})\rightarrow \cdots)\\
&\cong& \varinjlim(\cdots\rightarrow (I^{[p^e]}:\fm^\infty)/I^{[p^e]} \rightarrow (I^{[p^{e+1}]}:\fm^\infty)/I^{[p^{e+1}]}\rightarrow \cdots),
\end{eqnarray*}
and the map $(I^{[p^e]}:\fm^\infty)/I^{[p^e]} \rightarrow (I^{[p^{e+1}]}:\fm^\infty)/I^{[p^{e+1}]}$ is multiplication by $(x_1\cdots x_t)^{p^{e+1}-p^e}$. Let $\eta\in H^t_\fm(R)$, then there is an $e>0$ such that $\eta$ is represented by an element $r\in (I^{[p^e]}:\fm^\infty)/I^{[p^e]}$ in the direct limit system. By assumption, there then exists an $e'>0$ such that $r^{p^{e'}}\in (I^{[p^e]})^{[p^{e'}]}=I^{[p^{e+e'}]}$. Consider the following commutative diagram:
\[
\displaystyle
\begin{CD}
\frac{(I: \fm^{\infty})}{I} @>>> \frac{(I^{[p]}: \fm^{\infty})}{I^{[p]}} @>>> \cdots @>>> \frac{(I^{[p^e]}: \fm^{\infty})}{I^{[p^e]}} @>>> \cdots  \\
@VVF^{e'}_RV @VVF^{e'}_RV @.  @VVF^{e'}_RV\\
\frac{(I^{[p^{e'}]}: \fm^{\infty})}{I^{[p^{e'}]}} @>>> \frac{(I^{[p^{e'+1}]}: \fm^{\infty})}{I^{[p^{e'+1}]}} @>>> \cdots @>>> \frac{(I^{[p^{e+e'}]}: \fm^{\infty})}{I^{[p^{e+e'}]}} @>>> \cdots \\
\end{CD}
\]
The vertical maps $F^{e'}$ are raising elements to the $p^{e'}$ power and the horizontal maps $(I^{[p^{e''}]}:\fm^\infty)/I^{[p^{e''}]} \rightarrow (I^{[p^{{e''}+1}]}:\fm^\infty)/I^{[p^{{e''}+1}]}$ are multiplication by $(x_1\cdots x_t)^{p^{e''+1}-p^{e''}}$. The induced map on direct limit systems of the above commuting diagram gives the $e'$th Frobenius action on $H^{t}_\fm(R)$.  Since the representative $r\in (I^{[p^e]}:\fm^\infty)/I^{[p^e]}$ of $\eta$ is mapped to $0$ by $F^{e'}$, we see that the $e'$th iterated Frobenius action on $H^t_\fm(R)$ sends $\eta$ to $0$. 
\end{proof}

\begin{theorem}\label{characterization nilpotent 2} Let $(R,\fm, k)$ be a local ring of dimension $d$ and prime characteristic $p>0$ and let $t<d$. Then the following are equivalent:

\begin{enumerate}

\item $H^i_\fm(R)$ is F-nilpotent for all $ i\leq t$;

\item There exists a filter regular sequence $x_1, \ldots, x_t$ such that
$$(x_1^{p^{e}}, \ldots, x_s^{p^{e}}) : \fm^{\infty} \subseteq (x_1^{p^{e}}, \ldots, x_s^{p^{e}})^F$$ for all $s \le t$ and for all $e \ge 0$. 
\end{enumerate}

\end{theorem}

\begin{proof} $(1) \Rightarrow (2)$ follows from $(1) \Rightarrow (2)$ of Theorem \ref{characterization nilpotent 1}.\\
$(2) \Rightarrow (1)$ follows from Proposition \ref{frs imply nilpotent}.
\end{proof}
Comparing with the condition $(4)$ of Theorem~\ref{characterization nilpotent 1}, the condition $(2)$ of Theorem~\ref{characterization nilpotent 2} we need to consider all $s \le t$. However, we need only consider Frobenius powers of a filter regular sequence in $(2)$ of Theorem~\ref{characterization nilpotent 2} instead of mixed $p$th powers in $(4)$ of Theorem~\ref{characterization nilpotent 1}.

\section{\texorpdfstring{$F$}{F}-nilpotent rings}\label{Section F-nilpotent}
Recall that a local ring $(R,\fm,k)$ of dimension $d$ and prime characteristic $p>0$ is $F$-nilpotent if $H^i_\fm(R)$ are $F$-nilpotent whenever $i< d$ and $0^F_{H^d_\fm(R)}=0^*_{H^d_\fm(R)}$.

\begin{proposition}\label{tight closure is nilpotent} Let $(R, \fm,k)$ be an equidimensional excellent local ring of dimension $d$ and of prime characteristic $p>0$. If $x_1, \ldots, x_d$ is a filter regular sequence on $R$ and $((x_1, \ldots, x_d)^{[p^e]})^* = ((x_1, \ldots, x_d)^{[p^e]})^F$ for all $e \ge 0$ then the Frobenius action on $0^*_{H^d_{\fm}(R)}$ is nilpotent.
\end{proposition}
\begin{proof} Under the assumptions $R$ is equidimensional and excellent we have that for any $e\in\NN$
\[
0^*_{H^d_\fm(R)}=\varinjlim \frac{((x_1,\ldots x_d)^{[p^e]})^*}{(x_1,\ldots x_d)^{[p^e]}}.
\]
Moreover, it is generally the case that for each $e\in\NN$
\[
0^F_{H^d_\fm(R)}=\varinjlim \frac{((x_1,\ldots x_d)^{[p^e]})^F}{(x_1,\ldots x_d)^{[p^e]}}
\]
(see Remark~\ref{direct system}). The proposition easily follows from the above identifications of $0^*_{H^d_\fm(R)}$ and $0^F_{H^d_\fm(R)}$.
\end{proof}
Combination of Theorem~\ref{characterization nilpotent 2} and Proposition~\ref{tight closure is nilpotent} yields a sufficient criterion for a ring to be $F$-nilpotent.

\begin{theorem}\label{sufficient criteria for F-nilpotent} Let $(R, \fm,k)$ be an equidimensional excellent local ring of dimension $d$ and of prime characteristic $p>0$.  Suppose $x_1, \ldots, x_d$ is a filter regular sequence of $R$ such that the following hold:
\begin{enumerate}
\item For all $t<d$ and for all $e \ge 0$, $(x_1^{p^e}, \ldots, x_t^{p^e}): \fm^{\infty} \subseteq (x_1^{p^e}, \ldots, x_t^{p^e})^F$.
\item For all $e\geq 0$, $(x_1^{p^e}, \ldots, x_d^{p^e})^* = (x_1^{p^e}, \ldots, x_d^{p^e})^F$.
\end{enumerate}
Then $R$ is $F$-nilpotent.
\end{theorem} 

If $(R,\fm,k)$ satisfies colon capturing and if $x_1,\ldots, x_d$ is a system of parameters, then for all $t<d$ one has $(x_1,\ldots, x_t):\fm^\infty\subseteq (x_1,\ldots, x_t):x^{\infty}_{t+1}\subseteq (x_1,\ldots, x_t)^*$. If such a ring also satisfies $(x_1,\ldots, x_t)^F=(x_1,\ldots, x_t)^*$ for all filter regular sequences then $R$ must be $F$-nilpotent by Theorem~\ref{sufficient criteria for F-nilpotent}.

\begin{corollary}\label{tight imply nilpotent}  Let $(R, \fm, k)$ be an excellent equidimensional local ring of dimension $d$ and of prime characteristic $p>0$. Suppose $x_1, \ldots, x_d$ is a filter regular sequence satisfying that $(x_1^{p^e}, \ldots, x_t^{p^e})^* = (x_1^{p^e}, \ldots, x_t^{p^e})^F$ for all $t \le d$ and for all $e\ge 0$. Then $R$ is $F$-nilpotent.
\end{corollary}

We now discuss the notion of a relative tight closure of the $0$-submodule of a local cohomology module.

\begin{remark}
Let $I\subseteq K$ be ideals of $R$ and suppose $K/I$ is an ideal of $R/I$ generated by $t$-elements. Then by $(2)$ of Remark~\ref{direct system}  the tight closure of the zero submodule of $H^t_K(R/I)$ with respect to $R$ is
\[
0^{*_R}_{H^t_K(R/I)} = \{ \eta \in H^t_K(R/I) \mid c F^e_R(\eta) = 0 \in H^t_K(R/I^{[p^e]}) \text{ for some } c \in R^{\circ} \text{ and for all } e \gg 0\}.
\]
By the same remark,  the Frobenius closure of zero submodule of $H^t_K(R/I)$ with respect to $R$ is
\[0^{F_R}_{H^t_K(R/I)} = \{ \eta \in H^t_K(R/I) \mid F^e_R(\eta) = 0 \in H^t_K(R/I^{[p^e]}) \text{ for some } e \ge 0\}.
\]
\end{remark}

Similar to $(3)$ of Remark~\ref{direct system} we have the following.

\begin{lemma}\label{Direct limit of Frobenius closures lemma} Let $(R,\fm,k )$ be a local ring of prime characteristic $p>0$, $I\subseteq R$ an ideal, $x_1,\ldots, x_t$ a sequence of elements in $R$, and $K=(I,x_1,\ldots, x_t)$. If we identify 
\[
H^t_K(R/I)\cong \varinjlim_n\left(\cdots \to \frac{R}{(I,x_1^n,\ldots, x_t^n)}\xrightarrow{x_1\cdots x_t}\frac{R}{(I,x_1^{n+1},\ldots, x_t^{n+1})}\to \cdots\right)
\]
then 
\[
0^{F_R}_{H^t_K(R/I)}\cong \varinjlim_n \frac{(I,x_1^n, \ldots, x_t^n)^F}{(I,x_1^n, \ldots, x_t^n)}.
\]
\end{lemma}

\begin{proof}
It is easy to see that
\[
0^{F_R}_{H^t_K(R/I)}\supseteq \varinjlim_n \frac{(I,x_1^n, \ldots, x_t^n)^F}{(I,x_1^n, \ldots, x_t^n)}.
\]
Conversely, let $\eta\in 0^{F_R}_{H^t_K(R/I)}$. Without loss of generality we may assume $\eta$ is represented by an element $x+(I,x_{1},\ldots, x_t)$ in $R/(I,x_1,\ldots, x_t)$. We are assuming that there is an $e\in \NN$ such that $\eta$ is mapped to $0$ under $F^e: H^t_K(R/I)\to H^t_K(R/I^{[p^e]})$. If we identify $H^t_K(R/I^{[p^e]})$ as the direct limit system
\[
\varinjlim_n\left(\cdots \to \frac{R}{(I^{[p^e]},x_1^{p^en},\ldots, x_t^{p^en})}\xrightarrow{(x_1\cdots x_t)^{p^e}}\frac{R}{(I^{[p^e]},x_1^{p^e(n+1)},\ldots, x_t^{p^e(n+1)})}\to \cdots\right)
\]
then $\eta$ is mapped to the element represented by $x^{p^e}+(I^{[p^e]},x_1^{p^e},\ldots, x_t^{p^e})$. Since $F^e(\eta)=0$ there is a $n\in \NN$ such that $x^{p^e}(x_1\cdots x_t)^{p^e(n-1)}\in (I^{[p^e]},x_1^{p^en},\ldots, x_t^{p^en})$, hence $x(x_1\cdots x_t)^{n-1}$ is an element of $(I,x^n_1,\ldots, x_t^n)^F$ and $\eta$ is also represented by $x(x_1\cdots x_t)^{n-1}+(I,x_1^n,\ldots, x_t^n)$. 
\end{proof}

We also extend (4) of Remark~\ref{direct system} in the same way we extend (3) of Remark~\ref{direct system} in Lemma~\ref{Direct limit of Frobenius closures lemma}, but first we recall a couple of facts concerning annihilators of local cohomology modules.  Let $(R,\fm,k)$ be a local ring of dimension $d>0$ and for each $i \ge 0$ set $\fa_i(R) = \mathrm{Ann}(H^i_{\fm}(R))$ and let $\fa(R) = \fa_0(R) \ldots \fa_{d-1}(R)$. Suppose further that $R$ is an image of a Cohen-Macaulay local ring. Then we have the following (see \cite[Section 8.1]{BH93}): 
\begin{enumerate}
\item $\dim R/\fa(R) < d$.
\item If $x_1, \ldots, x_i$, $1 \le i \le d$ is part of a system of parameters then 
\[
\fa(R) \big( (x_1, \ldots, x_{i-1}): x_i\big) \subseteq (x_1, \ldots, x_{i-1}).
\]
\end{enumerate}
In particular, repeated application of (2) provides the following:

\begin{lemma}\label{ann and colon} Let $(R, \fm,k)$ be an excellent local ring of dimension $d>0$. Suppose that $x_1, \ldots, x_s, x_{s+1}, \ldots, x_{s+t}$ is a part of a system of parameters of $R$ then 
\begin{align*}
\fa(R)^t \big( (x_1^{n_1}, \ldots, x_{s}^{n_s}, x_{s+1}^{n_{s+1}+m_{s+1}}, \ldots, x_{s+t}^{n_{s+t}+m_{s+t}}): & (x_{s+1}^{m_{s+1}} \cdots x_{s+t}^{m_{s+t}})\big) \subseteq \\
& (x_1^{n_1}, \ldots, x_{s}^{n_s}, x_{s+1}^{n_{s+1}}, \ldots, x_{s+t}^{n_{s+t}}) 
\end{align*}
for all $n_i, m_i \ge 1$.
\end{lemma}

\begin{proof} Set $I = (x_1^{n_1}, \ldots, x_{s}^{n_s})$, let $a_1,\ldots,a_t\in \fa(R)$, and let 
\[
\gamma\in (I,x_{s+1}^{n_{s+1}+m_{s+1}}, \ldots, x_{s+t}^{n_{s+t}+m_{s+t}}): (x_{s+1}^{m_{s+1}}\cdots x_{s+t}^{m_{s+t}}).
\]
Then there is an $r\in R$ so that  $\gamma x_{s+1}^{m_{s+1}}\cdots x_{s+t-1}^{m_{s+t-1}}-rx_{s+t}^{n_{s+t}}\in (I,x_{s+1}^{n_{s+1}},\ldots, x_{s+t-1}^{n_{s+t-1}}):x^{m_{s+t}}_{s+t}.$ Therefore 
\[
a_1\gamma\in (I,x_{s+1}^{n_{s+1}+m_{s+1}}, \ldots, x_{s+t-1}^{n_{s+t-1}+m_{s+t-1}},x_{s+t}^{n_{s+t}}): (x_{s+1}^{m_{s+1}}\cdots x_{s+t-1}^{m_{s+t-1}}).
\]
By induction, $a_t\cdots a_2 a_1\gamma \in (x_1^{n_1}, \ldots, x_{s}^{n_s}, x_{s+1}^{n_{s+1}}, \ldots, x_{s+t}^{n_{s+t}}).$
\end{proof}

\begin{lemma}\label{Direct limit of tight closures lemma}  Let $(R, \fm, k)$ be an excellent equidimensional local ring of prime characteristic $p>0$. Suppose that $x_1, \ldots, x_s, x_{s+1}, \ldots, x_{s+t}$ is a part of system of parameters of $R$, and set $I = (x_1, \ldots, x_s)$ and $K = (I,x_{s+1}, \ldots, x_{s+t})$. If we identify 
\[
H^t_K(R/I)\cong \varinjlim_n\left(\cdots \to \frac{R}{(I,x_{s+1}^n,\ldots, x_{s+t}^n)}\xrightarrow{x_{s+1}\cdots x_{s+t}}\frac{R}{(I,x_{s+1}^{n+1},\ldots, x_{s+t}^{n+1})}\to \cdots\right)
\]
then
\[
0^{*_R}_{H^t_K(R/I)} \cong \varinjlim_n \frac{(I, x_{s+1}^n, \ldots, x_{s+t}^n)^*}{(I, x_{s+1}^n, \ldots, x_{s+t}^n)}.
\]
\end{lemma}

\begin{proof} 
It easy to see that
\[
0^{*_R}_{H^t_K(R/I)}\supseteq  \varinjlim_n \frac{(I, x_{s+1}^n, \ldots, x_{s+t}^n)^*}{(I, x_{s+1}^n, \ldots, x_{s+t}^n)}.
\]
Now suppose $\eta\in 0^{*R}_{H^t_K(R/I)}$. Without loss of generality we may assume that $\eta$ is represented by an element $x+(I,x_{s+1},\ldots, x_{s+t})\in \frac{R}{(I,x_{s+1},\ldots, x_{s+t})}$. The Frobenius action 
\[
F^e: H^t_{K}(R/I)\to H^t_K(R/I^{[p^e]})\cong \varinjlim_n \frac{R}{(I^{[p^e]},x^{p^en}_{s+1},\ldots, x_{s+t}^{p^en})}
\]
 maps $\eta$ to the element $\eta^{p^e}$ represented by $x^{p^e}+(I^{[p^e]},x_{s+1}^{p^e},\ldots,x_{s+t}^{p^e})\in \frac{R}{(I^{[p^e]},x_{s+1}^{p^e},\ldots,x_{s+t}^{p^e})}.$
 We are assuming there is $c\in R^\circ$ such that $c\eta^{p^e}=0$ for all $e\gg 0$. So for each $e\gg0$ there is an $n\in \NN$ such that $cx^{p^e}(x_{s+1}\cdots x_{s+t})^{p^e(n-1)}\in(I^{[p^e]},x_{s+1}^{p^en},\ldots, x_{s+t}^{p^en})$. We are assuming $R$ is equidimensional, hence there is $d \in R^{\circ} \cap \fa(R)^t$ and it follows
 \[
 dcx^{p^e}\in (I^{[p^e]},x_{s+1}^{p^e},\ldots, x_{s+t}^{p^e})=(I,x_{s+1},\ldots, x_{s+t})^{[p^e]}
 \]
for all $e \gg 0$ by Lemma~\ref{ann and colon}. Therefore $\eta$ is represented by an element of $\frac{(I,x_s+1,\ldots,x_{s+t})^*}{(I,x_{s+1},\ldots, x_{s+t})}.$
\end{proof}

\begin{remark}\label{Remark in section 5} Let $(R, \fm,k)$ be an equidimensional excellent local ring of dimension $d$ and of characteristic $p>0$. Let $x_1, \ldots, x_i$ be a part of system of parameters of $R$, and set $I = (x_1, \ldots, x_i)$. Then
\begin{enumerate}
\item $0^{*_R}_{H^{d-i}_{\fm}(R/I)}$ agrees with the usual tight closure $0^{*}_{H^{d-i}_{\fm}(R/I)}$, here we consider $H^{d-i}_{\fm}(R/I)$ as an $R$-module. Indeed, set $K = (x_{i+1}, \ldots, x_d)$ a parameter ideal of $R/I$. We have $H^{d-i}_{\fm}(R/I) \cong H^{d-i}_K(R/I)$. Thus 
\begin{align*}
H^{d-i}_{\fm}(R/I) \otimes_R F^e_*(R) \cong H^{d-i}_K(R/I) \otimes_R F^e_*(R) &\cong H^{d-i}_K(F^e_*(R)/IF^e_*(R))\\
& \cong H^{d-i}_{\fm}(F^e_*(R)/IF^e_*(R)).
\end{align*}
Therefore if we identify $F^e_*(R)$ with $R$, then the map $H^{d-i}_{\fm}(R/I) \xrightarrow{\mathrm{id} \otimes F^e} H^{d-i}_{\fm}(R/I) \otimes_R F^e_*(R)$ can be identified with the $e$-th Frobenius action with respect to $R$, $F^e_R: H^{d-i}_{\fm}(R/I) \to H^{d-i}_{\fm}(R/I^{[p^e]})$. Now two tight closures $0^{*_R}_{H^{d-i}_{\fm}(R/I)}$ and $0^{*}_{H^{d-i}_{\fm}(R/I)}$ are the same by their definitions.
\item Let $\fa_i(R)=\Ann(H^i_\fm(R))$ and let $\fa(R)=\fa_0(R)\cdots \fa_{d-1}(R)$. Notice that since $\dim R/\fa (R) < d$ and $R$ is equidimensional that $R^\circ\cap \fa(R)\neq 0$. Moreover for any part of system of parameters $y_1, \ldots, y_i$ of $R$ we have $\fa (R) H^j_{\fm}(R/(y_1, \ldots, y_i))=0$ for all $j < d-i$ (see \cite[Remark 2.2, Lemma~3.7]{CQ16}).  Therefore for all $\eta \in  H^{j}_{\fm}(R/I)$, $j < d-i$, we have $\fa (R) F^e_R( \eta) = 0 \in H^j_{\fm}(R/I^{[p^e]})$ for all $e \ge 0$. Thus $0^{*_R}_{H^{j}_{\fm}(R/I)} = H^{j}_{\fm}(R/I)$  for all $j < d-i$. 
\end{enumerate}
\end{remark}
By the above remark, an equidimensional excellent local ring $(R, \fm,k)$ is $F$-nilpotent if and only if $0^{*_R}_{H^{j}_{\fm}(R)} = 0^{F_R}_{H^{j}_{\fm}(R)}$  for all $j \le d$.

Let $(R,\fm,k)$ be a local ring of prime characteristic $p>0$ and $I\subseteq R$ an ideal. Suppose that $\dim(R/I)=t$. We say that $R/I$ is $F$-nilpotent with respect to $R$ if the following hold:
\begin{enumerate}
\item $H^i_\fm(R/I)$ is $F$-nilpotent with respect $R$ for each $i<t$.
\item $0^{*_R}_{H^t_\fm(R/I)}=0^{F_R}_{H^t_\fm (R/I)}$.
\end{enumerate}

\begin{theorem}\label{general deformation section 5} Let $(R,\fm,k)$ be an excellent equidimensional local ring of dimension $d$ and prime characteristic $p>0$. Let $x_1, \ldots, x_i$, $i \le d$, be a filter regular sequence of $R$ and let $I = (x_1, \ldots, x_{i-1})$. The following are equivalent
\begin{enumerate} 
\item $R/I^{[p^e]}$ is $F$-nilpotent with respect to $R$ for all $e \ge 0$.
\item $R/(I^{[p^e]},x_i^{p^{e'}})$ is $F$-nilpotent with respect to $R$ for all $e, e' \ge 0$.
\end{enumerate}
\end{theorem}
\begin{proof} $(1) \Rightarrow (2)$ It is enough to prove that $R/(I,x_i)$ is $F$-nilpotent with respect to $R$. Moreover by Theorem~\ref{general Deformation} we need only to prove that $0^{*_R}_{H^{d-i}_{\fm}(R/(I,x_i))}$ is $F$-nilpotent with respect to $R$. Similar to the proof of Theorem \ref{general Deformation} we have the following commutative diagram
\[
\begin{CD}
H^{d-i}_{\fm}(R/I) @>>> H^{d-i}_{\fm}(R/(I,x_i)) @>\delta>> H^{d-i+1}_{\fm}(R/I)  \\
@VVF^{e}_RV @VVF^{e}_RV @VVF^{e}_RV \\
H^{d-i}_{\fm}(R/I^{[p^e]})  @>\beta>> H^{d-i}_{\fm}(R/(I,x_i)^{[p^e]}) @>>>  H^{d-i+1}_{\fm}(R/I^{[p^e]}). \\
@VVF^{e'}_RV @VVF^{e'}_RV  \\
H^{d-i}_{\fm}(R/I^{[p^{e+e'}]})  @>>> H^{d-i}_{\fm}(R/(I,x_i)^{[p^{e+e'}]}). \\
\end{CD}
\]
For any element $\eta \in 0^{*_R}_{H^{d-i}_{\fm}(R/(I,x)}$ we have $\delta(\eta) \in  0^{*_R}_{H^{d-i+1}_{\fm}(R/I)}$ by Lemma \ref{Direct limit of tight closures lemma} and chasing the image of $\eta$. Therefore we can choose $e$ large enough such that $F^e_R(\delta(\eta)) = 0$. Thus $F^e_R(\eta) \in \mathrm{Im}(\beta)$. On the other hand $H^{d-i}_{\fm}(R/I^{[p^e]})$ is $F$-nilpotent with respect to $R$, so we have $F^{e+e'}_R(\eta) = 0$ for all $e' \gg 0$. Hence $0^{*_R}_{H^{d-i}_{\fm}(R/(I,x))} = 0^{F_R}_{H^{d-i}_{\fm}(R/(I,x))}$, and $R/(I,x_i)$ is $F$-nilpotent.

$(2) \Rightarrow (1)$ It is enough to prove that $R/I$ is $F$-nilpotent. Moreover by Theorem \ref{general Deformation} we need only to prove that $0^{*_R}_{H^{d-i+1}_{\fm}(R/I)}$ is $F$-nilpotent with respect to $R$. Extend $x_1, \ldots,x_i$ to a full system of parameters $x_1, \ldots, x_d$ and let $\eta \in 0^{*_R}_{H^{d-i+1}_{\fm}(R/I)}$. By Lemma \ref{Direct limit of tight closures lemma} and by replacing $x_i, \ldots,x_d$ by their $p^e$-powers, $e \gg 0$, we can assume that $\eta$ is represented by an element $x+(I,x_{i},\ldots, x_{d})\in \frac{R}{(I,x_{i},\ldots, x_{d})}$, where $x \in (I, x_i, \ldots, x_d)^*$. This element maps to some element $\theta \in 0^{*_R}_{H^{d-i}_{\fm}(R/(I,x_i))}$. Hence in the following commutative diagram $\delta(\theta) = \eta$
\[
\begin{CD}
 H^{d-i}_{\fm}(R/(I,x_i)) @>\delta>> H^{d-i+1}_{\fm}(R/I)  \\
 @VVF^{e}_RV @VVF^{e}_RV \\
H^{d-i}_{\fm}(R/(I,x_i)^{[p^e]}) @>>>  H^{d-i+1}_{\fm}(R/I^{[p^e]}). 
\end{CD}
\]
Since $R/(I,x_i)$ is $F$-nilpotent, $F^e_R(\theta) = 0$ for large enough $e$. So $F^e_R(\eta) = 0$, and hence $R/I$ is $F$-nilpotent.
\end{proof}

An application of the previous theorem is the following deformation type result for $F$-nilpotent rings.

\begin{theorem} Let $(R,\fm,k)$ be an excellent equidimensional local ring of dimension $d$ and prime characteristic $p>0$. The following are equivalent:
\begin{enumerate}
\item The ring $R$ is $F$-nilpotent.
\item For each filter regular element $x$ the ring $R/(x)$ is $F$-nilpotent with respect to $R$.
\item There exists a filter regular element $x$ such that for each $n\in \NN$ the ring $R/(x^n)$ is $F$-nilpotent with respect to $R$.
\item There exists a filter regular element $x$ such that for each $e\in \NN$ the ring $R/(x^{p^e})$ is $F$-nilpotent with respect to $R$.
\end{enumerate}
\end{theorem}

We next give a proof of Theorem~\ref{Main theorem frobenius closure}.

\begin{theorem}\label{characterization F-nil 1} Let $(R, \fm,k)$ be an excellent local ring of dimension $d$ and of prime characteristic $p>0$. Consider the following statements:
\begin{enumerate}
\item $R$ is $F$-nilpotent.
\item For every parameter ideal $\fq$ we have $\fq^* = \fq^F$.
\item There exists filter regular sequence $x_1, \ldots, x_d$ such that 
\[
(x_1^{n_1}, \ldots, x_d^{n_d})^* = (x_1^{n_1}, \ldots, x_d^{n_d})^F
\]
 for all $n_1, \ldots, n_d \ge 1$.
\item There exists filter regular sequence $x_1, \ldots, x_d$ such that 
\[
(x_1^{p^{e_1}}, \ldots, x_d^{p^{e_d}})^* = (x_1^{p^{e_1}}, \ldots, x_d^{p^{e_d}})^F
\]
for all $e_1, \ldots, e_d \ge 0$.
\end{enumerate}
Then $(1)\Rightarrow (2)\Rightarrow (3)\Rightarrow (4)$. If $R$ is equidimensional then $(4)\Rightarrow (1)$.
\end{theorem}

\begin{proof} 
 The implications $(2)\Rightarrow (3)\Rightarrow (4)$ are obvious. We first prove that $(1)\Rightarrow (2)$.  Suppose $R$ is $F$-nilpotent and $\fq=(x_1,\ldots, x_d)$ is a parameter ideal. We may assume $x_1,\ldots, x_d$ is a filter regular sequence.  Note that $R$ is equidimensional by (3) of Proposition~\ref{Basic Proposition about F-nilpotent Rings}. Applying Theorem~\ref{general deformation section 5} consecutively we have $R/(x_1), \ldots, R/(x_1, \ldots, x_d)$ are $F$-nilpotent with respect to $R$. Moreover $\dim R/\fq = 0$ and $H^0_{\fm}(R/\fq) = R/\fq$. We also have $0^{*_R}_{R/\fq} = \fq^*/\fq$ and $0^{F_R}_{R/\fq} = \fq^F/\fq$. Therefore $\fq^* = \fq^F$ for all parameter ideals $\fq$.
 
We prove the implication $(4) \Rightarrow (1)$ under the additional assumption that $R$ equidimensional. Fix $e_1,\ldots,e_{d-1}\in \NN$ and consider the filter regular sequence $x^{p^{e_1}}_1,\ldots, x^{p^{e_{d-1}}}_{d-1}$ of $R$ and let $I_{e_1, \ldots, e_{d-1}}=(x^{p^{e_1}}_1,\ldots,x^{p^{e_{d-1}}}_{d-1})$. Then $R/(I_{e_1, \ldots, e_{d-1}}^{[p^e]},x_d^{e'})$ is $F$-nilpotent relative to $R$ for all $e, e' \ge 0$. By Theorem~\ref{general deformation section 5}, $R/I_{e_1, \ldots, e_{d-1}}$ is $F$-nilpotent relative to $R$ for all $e_1, \ldots, e_{d-1} \ge 0$. Consecutive  use of Theorem~\ref{general deformation section 5} then shows $R$ is $F$-nilpotent. 
\end{proof}

We now wish to show that if $(R,\fm,k)$ is a local ring of prime characteristic $p>0$ which is $F$-nilpotent and $I$ an ideal generated by part of system of parameters then $I^*=I^F$. But first, we will need to discuss the notion of Frobenius test exponent for parameter ideals.

\begin{definition} Let $(R, \fm, k)$ be a local ring of prime characteristic $p>0$. The Frobenius test exponent for parameter ideals of $R$ defined as follows
\[
Fte(R) = \min \{e \mid (\fq^{F})^{[p^e]} = \fq^{[p^e]} \text{ for all parameter ideals } \fq\},
\]
and $Fte(R) = \infty$ if no such $e$ exists.
\end{definition}

\begin{remark} Katzman and Sharp asked in \cite{KS06} under what hypotheses on local ring $(R,\fm,k)$ of prime characteristic $p>0$ is $Fte(R) < \infty$.  If $R$ is Cohen-Macaulay they showed that $Fte(R)$ is finite and is equal to the Hartshorne-Spieser-Lyubeznik number of $R$, which under the Cohen-Macaulay hypothesis, is the least integer $e$ for which $F^e:0^F_{H^d_\fm(R)}\to 0^F_{H^d_\fm(R)} $ is the $0$-map. The authors of \cite{HKSY06} are able to show under the weaker hypotheses $R$ is generalized Cohen-Macaulay that $Fte(R)<\infty$. 
\end{remark}
The second author of this paper uses techniques of this article in \cite{Q18} to provide a simpler proof of the main result of \cite{HKSY06}. An argument, which also utilizes the techniques of this paper, also proves the Frobenius test exponent for parameter ideals of an $F$-nilpotent ring is finite.

\begin{theorem}[\cite{Q18}]\label{fte theorem} Let $(R,\fm,k)$ be a local ring of dimension $d$ and prime characteristic $p>0$. If $H^i_\fm(R)$ is $F$-nilpotent for all $i<d$ then $Fte(R) < \infty$.
\end{theorem}

Theorem~\ref{characterization F-nil 1} and Theorem~\ref{fte theorem} yield the desired corollary.

\begin{corollary}\label{Frobenius closure partial parameter ideal} Let $(R, \fm,k)$ be an excellent $F$-nilpotent local ring of dimension $d$. Then for every part of system of parameters $x_1, \ldots, x_t$ of $R$ we have $(x_1, \ldots, x_t)^* = (x_1, \ldots, x_t)^F$.
\end{corollary}
\begin{proof} Extend $x_1, \ldots, x_t$ to a full system of parameters $x_1, \ldots, x_d$. We can assume that $R$ is reduced. Let $x\in (x_1,\ldots, x_t)^*$, then $x \in (x_1, \ldots, x_t, x_{t+1}^n, \ldots, x_d^n)^F$ for all $n \ge 1$. Let $e = Fte(R)$, then 
\[
x^{p^e} \in \bigcap_{n \ge 1}(x_1, \ldots, x_t, x_{t+1}^n, \ldots, x_d^n)^{[p^e]} = (x_1, \ldots, x_t)^{[p^e]}
\]
by the Krull interestion theorem. Thus we have $x \in (x_1, \ldots, x_t)^F$. 
\end{proof}
Combining the above result with Corollary \ref{tight imply nilpotent} we have the following.
\begin{theorem}\label{Characterization of F-nil 2} Let $(R, \fm, k)$ be an equidimensional excellent local ring of dimension $d$ and of prime characteristic $p>0$. Then the following are equivalent:
\begin{enumerate}
\item $R$ is $F$-nilpotent.
\item There exists filter regular sequence $x_1, \ldots, x_t$ so that  $(x_1^{p^{e}}, \ldots, x_t^{p^{e}})^F = (x_1^{p^{e}}, \ldots, x_t^{p^{e}})^*$
for all $t \le d$ and for all $e \ge 0$.
\end{enumerate}
\end{theorem}

Srinivas and Takagi prove that the property of being $F$-nilpotent localizes for rings which are assumed to be $F$-finite, \cite[Proposition~2.4]{ST17}. All $F$-finite rings are known to be excellent, \cite[Theorem~2.5]{K76}. Theorem~\ref{Characterization of F-nil 2} gives a method of showing $F$-nilpotence localizes for all prime characteristic rings which are excellent.

\begin{corollary}\label{F nilpotent localizes} Let $(R,\fm,k)$ be an excellent local ring of prime characteristic $p>0$. If $R$ is $F$-nilpotent then $R_P$ is $F$-nilpotent for each $P\in\Spec(R)$.
\end{corollary}
\begin{proof}
By Proposition~\ref{Basic Proposition about F-nilpotent Rings} we may assume $R$ is reduced. Suppose $\Ht(P)=t$ and $I=(a_1,\ldots, a_t)R_P$ a parameter ideal of $R_P$. Following the proof of \cite[Proposition~6.9]{QS17} we can choose partial parameter sequence $x_1,\ldots,x_t$ of $R$ such that $I=(x_1,\ldots, x_t)R_P$. It is well-known that tight closure commutes with localization for parameter ideals under our hypotheses, see \cite[Theorem~8.1]{AHH97} and \cite[Theorem~5.1]{S94}. The Frobenius closure of an ideal $I$ is the extension and contraction of the ideal along a high enough iterate of the Frobenius endomorphism. Hence the Frobenius closure of any ideal commutes with localization. Therefore for each $e\in \NN$ we have by Theorem~\ref{Characterization of F-nil 2} that
\[
I^F=((x_1,\ldots,x_t)R_P)^F=(x_1,\ldots,x_t)^FR_P=(x_1,\ldots,x_t)^*R_P=I^*,
\]
where the third equality is follows from Corollary~\ref{Frobenius closure partial parameter ideal}. Therefore the tight closure of every parameter ideal of $R_P$ is equal to its Frobenius closure and the ring $R_P$ is now seen to be $F$-nilpotent by Theorem~\ref{characterization F-nil 1}.
\end{proof}

\specialsection*{Acknowledgments}

This project resulted from a trip the second named author took to the University of Utah, we would like to thank Linquan Ma and Karl Schwede for making that trip possible.

\bibliographystyle{alpha}
\bibliography{References}

\end{document}